\documentclass[12pt, letter]{amsart}
\usepackage{amssymb,latexsym%,a4wide
}
\usepackage[dvips,centering,includehead,width=15.1cm,height=22.7cm]{geometry}

\def\RR{{\mathbb R}}

\def\poly{\textrm{poly}}
\def\aa{{\varepsilon}}
\def\Pcal{\mathcal{P}}

\newtheorem{theorem}{Theorem}[section]
\newtheorem{proposition}[theorem]{Proposition}

\newtheorem{corollary}[theorem]{Corollary}
\newtheorem{lemma}[theorem]{Lemma}

\theoremstyle{definition}

\newtheorem{remark}[theorem]{Remark}

\newtheorem{example}[theorem]{Example}
\newtheorem{definition}[theorem]{Definition}
\newtheorem{problem}[theorem]{Problem}

\numberwithin{equation}{section}

\begin{document}

\title[Subtraction-free complexity]{Subtraction-free complexity,\\[.05in]
cluster transformations,
and spanning trees}

\author{Sergey Fomin}
\address{Department of Mathematics, University of Michigan, 530 Church
  Street, Ann Arbor, MI 48109-1043, USA}
\email{fomin@umich.edu}
 \urladdr{http://www.math.lsa.umich.edu/$\tilde{\ }$fomin/}

\author{Dima Grigoriev}
\address{CNRS, Math\'ematiques, Universit\'e de Lille, Villeneuve
  d'Ascq, 59655, France}
\email{Dmitry.Grigoryev@math.univ-lille1.fr}
\urladdr{http://en.wikipedia.org/wiki/Dima$\underbar{\ }$Grigoriev}

\author{Gleb Koshevoy}
\address{Central Institute of Economics and Mathematics, Nahimovskii
  pr.~47, Moscow 117418, Russia}
\email{koshevoy@cemi.rssi.ru}
\urladdr{http://mathecon.cemi.rssi.ru/en/koshevoy/}

\thanks{Communicated by Peter B\"urgisser}


\thanks{We thank the Max-Planck Institut f\"ur Mathematik %(MPIM)
for its hospitality during the writing of this paper.
Partially supported by NSF grant DMS-1101152 (S.~F.),
RFBR/CNRS grant 10-01-9311-CNRSL-a,
%the RFFI grant 02-01-00093,
and MPIM (G.~K.).}

\date{Submitted August 27, 2013. Revised September 22, 2014.}

\begin{abstract}
Subtraction-free computational complexity is the version of arithmetic
\hbox{circuit} complexity
that allows only three operations: addition,
multiplication, and~division.

We use cluster transformations to design efficient subtraction-free algorithms
for computing Schur functions and their skew, double, and
supersymmetric \hbox{analogues,}
thereby generalizing earlier results by P.~Koev.

We develop such algorithms for computing generating functions of spanning trees,
both directed and undirected.
A comparison to the lower bound due to \hbox{M.~Jerrum} and M.~Snir
shows that in subtraction-free computations,
``division can be exponentially powerful.''

Finally, we give a simple example where the gap between ordinary and subtraction-free complexity is exponential.
\end{abstract}


\ \vspace{-0.2in}

\maketitle

\tableofcontents

\section*{Introduction}

This paper is motivated by the problem of dependence of
algebraic complexity %of a rational function
on the set of allowed operations.
Suppose that a rational function~$f$ can in principle be computed
using a restricted set of arithmetic operations
$M\subset\{+,-,*,/\}$;
how does the complexity of~$f$ (i.e., the minimal number of steps in
such a computation) depend on the choice~of~$M$?
For example, let $f$ be a polynomial with nonnegative coefficients;
%say some kind of a generating function.
then it can be computed
%using addition and multiplication only;
%or we could allow division as well---but not subtraction
without using subtraction
(we call this a \emph{subtraction-free} computation).
Could this restriction dramatically alter the complexity of~$f$?
What if we also forbid using division?

%\medskip

One natural test %case for this question
is provided by the
\emph{Schur functions} and their various generalizations. \pagebreak[3]
Combinatorial descriptions of these polynomials are quite
complicated, and the (non\-negative) coefficients in their monomial expansions
are known to be hard to compute. \pagebreak[3]
On~the other hand, well-known determinantal formulas for
Schur functions yield fast (but not subtraction-free) algorithms for
computing them.

In fact, one \emph{can} compute a Schur function in
polynomial time without using sub\-traction.
An outline of such an algorithm was first proposed by P.~Koev~\cite{Koev-2007}
in~2007.
In this paper, we describe an alter\-native algorithm utilizing
the machinery of \emph{cluster transformations},
a family of subtraction-free rational maps that play a key
role in the theory of cluster algebras~\cite{FZ}.
We then further develop this approach to obtain
subtraction-free polynomial algorithms for computing
\emph{skew}, \emph{double}, and \emph{supersymmetric} Schur functions.

%\medskip

We also look at another natural class of polynomials:
the generating functions of spanning trees
(either directed or undirected)
in a connected (di)graph with weighted edges.
We use \emph{star-mesh transformations} to develop subtraction-free
algorithms that compute these generating functions in polynomial time.
In the directed case, this sharply contrasts with the exponential lower bound
due to M.~Jerrum and M.~Snir~\cite{Jerrum-Snir}
who showed that if one only allows additions and multiplications
(but no subtractions or divisions),
then the arithmetic circuit complexity of the generating
function for directed spanning~trees in an \hbox{$n$-vertex} complete digraph
grows exponentially in~$n$.
We thus obtain an \linebreak[3]
expo\-nential gap between subtraction-free and semiring complexity,
which can be informally expressed by saying that in the absence of subtraction,
division can be ``exponentially powerful''
(cf.\ L.~Valiant's result~\cite{Valiant} on the power of subtraction).
Recall that if subtraction~is allowed,
then division gates can be eliminated at polynomial cost,
as shown by V.~Strassen~\cite{Strassen}.
One could say that forbidding subtraction can dramatically
increase the power of division.

Jerrum and Snir \cite{Jerrum-Snir} have shown that their exponential lower bound
also holds in the \emph{tropical semiring}
$(\RR,+,\min)$
(see, e.g., \cite[Section~8.5]{Litvinov} and references therein).
Since our algorithms extend straightforwardly into the tropical setting,
we conclude that the circuit complexity of the
\emph{minimum cost arborescence} problem
drops from exponential to polynomial as one passes from the tropical semiring
to the \emph{tropical semifield} $(\RR,+,-,\min)$.

%\medskip

At the end of the paper, we present a simple example of a rational function $f_n$
%(indeed, a quadratic univariate polynomial)
whose ordinary circuit complexity is linear in~$n$ whereas its subtraction-free complexity,
while finite, grows at least exponentially in~$n$.

%\medskip

The paper is organized as follows.
Section~\ref{sec:complexity} reviews basic prerequisites in
algebraic complexity, along with some relevant historical background.
In Section~\ref{sec:main-results} we present our main results.
Their proofs %(i.e., the description and justification of the corresponding algorithms)
occupy the rest of the paper.
Sections~\ref{sec:main-algorithm}--\ref{sec:skew}
are devoted to subtraction-free algorithms for computing Schur functions and their
variations, while in
Sections~\ref{sec:spanning-trees}--\ref{sec:directed-spanning-trees}
we develop such algorithms
for computing generating functions for spanning trees, either
ordinary or directed.
In Section~\ref{sec:quadratic}, we demonstrate the existence of exponential gaps
between ordinary and subtraction-free complexity.

\pagebreak[3]

\section{Computational complexity}
\label{sec:complexity}

%\subsection{Arithmetic circuits}
We start by reviewing the relevant basic notions of computational complexity,
more specifically complexity of arithmetic circuits (with
restrictions).
See \cite{BCS, G, SY} for in-depth-treatment and further
references.

An \emph{arithmetic circuit} is an oriented network each of whose
nodes (called \emph{gates}) performs a single arithmetic operation:
addition, subtraction, multiplication, or division. The circuit inputs
a collection of \emph{variables} (or indeterminates) as well as some
scalars, and outputs a rational function in those variables.
The \emph{arithmetic circuit complexity} of a rational function is the
smallest size of an arithmetic circuit that computes this function.

The following disclaimers further clarify the setup considered in this
paper:
\begin{itemize}
\item[-]
we define complexity as the number of gates in a circuit rather than
its depth;
\item[-]
we do not concern ourselves with parallel computations;
%\item[-]
%we do not pay attention to bit complexity; so the size of constants
%does not matter;
\item[-]
we allow arbitrary positive integer scalars as inputs.
\end{itemize}
%(Alternative version: only 0,1 are given. Another alternative: allow rational scalars.)
%\item[-]
%We ignore the issue of space/memory used by our algorithms, focusing
%exclusively on their time/circuit complexity.
Although we focus on arithmetic circuit complexity,
we also provide %(very similar, and rather straightforward)
\emph{bit complexity} estimates for our algorithms.
For the latter purpose, the input variables should be viewed as numbers rather than
formal variables.

As is customary in complexity theory, we consider families of
computational problems indexed by a positive integer parameter~$n$,
%(or several such parameters),
and only care about the rough
asymptotics of the arithmetic complexity as a function of~$n$. The
number of variables may depend on~$n$.

Of central importance is the dichotomy between polynomial and
superpolynomial (in particular exponential) complexity classes. We use
the shorthand $\poly(n)$ to denote the dependence of complexity on $n$
that can be bounded from above by a polynomial in~$n$.

Perhaps the most important (if simple) example of a sequence
of functions whose arithmetic circuit complexity is $\poly(n)$ is the
determinant of an $n$ by $n$ matrix. (The entries of a matrix are
treated as indeterminates.) The simplest---though not the
fastest---polynomial algorithm for computing the determinant is
Gaussian elimination.

%\subsection{Restricting the set of operations}

In this paper, we are motivated by the following fundamental question:
How does the complexity of an algebraic expression depend on the set
of operations allowed?

Let us formulate the question more precisely.
Let $M$ be a subset of the set $\{+,-,*,/\}$ of arithmetic operations.
Let $Z\{M\}=Z\{M\}(x,y,\dots)$ denote the class of rational functions
in the variables $x,y,\dots$ which can be defined using only
operations in~$M$.
For example, the class $Z\{+,*,/\}$ consists of \emph{subtraction-free
  expressions}, i.e., those rational functions which can be written
without using subtraction (note that negative scalars are not allowed
as inputs).
To illustrate, $x^2-xy+y^2\in Z\{+,*,/\}(x,y)$ because
$x^2-xy+y^2=(x^3+y^3)/(x+y)$.

While the class $Z\{M\}$ can be defined for each of the $2^4=16$
subsets $M\subset\{+,-,*,/\}$,
there are only 9 distinct classes among these~16.
This is because addition can be emulated by subtraction: $x+y = x-((x-y)-x)$.
Similarly, multiplication can be emulated by division.
This leaves 3 essentially distinct possibilities for the additive
(resp., multiplicative) operations.
The corresponding 9 computational models are shown in
Table~\ref{table:9-models}.

\pagebreak[3]

\begin{table}[ht]
\begin{center}
\begin{tabular}{|c||c|c|c|}
\hline
% & $\emptyset$ & * & *,/ \\
& no multiplicative & multiplication & multiplication\\
& operations & only & and division \\
\hline
\hline
\begin{tabular}{c}no additive \\ operations\end{tabular}
& scalars & monomials & Laurent monomials\\
\hline
\begin{tabular}{c}addition only\end{tabular}
& \begin{tabular}{c}nonnegative linear\\ combinations\end{tabular}
& \begin{tabular}{c}nonnegative\\ polynomials\end{tabular}
& \begin{tabular}{c}subtraction-free\\ expressions\end{tabular}
\\
\hline
\begin{tabular}{c}addition and\\ subtraction\end{tabular}
& \begin{tabular}{c}linear\\ combinations\end{tabular}
& polynomials
& rational functions\\
\hline
\end{tabular}
\vspace{.1in}
\end{center}
\caption{Rational functions computable with restricted set of operations}
\label{table:9-models}
\end{table}

%\ \vspace{-.3in}

For each subset of arithmetic operations $M\subset\{+,-,*,/\}$, there
is the corresponding notion of (arithmetic circuit)
$M$-\emph{complexity} (of an element of $Z\{M\})$.
The interesting cases are those where both additive and multiplicative operations appear,
see Table~\ref{table:nontrivial-M}.

\begin{table}[ht]
\begin{center}
ordinary complexity\\[.05in]
\setlength{\unitlength}{1pt}
\begin{picture}(140,100)(10,0)
\thicklines

\put(50,0){\line(1,0){40}}
\put(50,0){\line(0,1){20}}
\put(90,20){\line(-1,0){40}}
\put(90,20){\line(0,-1){20}}
\put(70,10){\makebox(0,0){$+\ *$}}

\put(0,40){\line(1,0){40}}
\put(0,40){\line(0,1){20}}
\put(40,60){\line(-1,0){40}}
\put(40,60){\line(0,-1){20}}
\put(20,50){\makebox(0,0){$+\, -\, *$}}

\put(-35,50){\makebox(0,0){\begin{tabular}{c}
ring\\
complexity
\end{tabular}}}

\put(100,40){\line(1,0){40}}
\put(100,40){\line(0,1){20}}
\put(140,60){\line(-1,0){40}}
\put(140,60){\line(0,-1){20}}
\put(120,50){\makebox(0,0){$+\, *\, /$}}

\put(185,50){\makebox(0,0){\begin{tabular}{c}
subtraction-free\\
complexity
\end{tabular}}}

\put(45,80){\line(1,0){50}}
\put(45,80){\line(0,1){20}}
\put(95,100){\line(-1,0){50}}
\put(95,100){\line(0,-1){20}}
\put(70,90){\makebox(0,0){$+ - * \ /$}}

\put(55,20){\line(-1,1){20}}
\put(85,20){\line(1,1){20}}
\put(35,60){\line(1,1){20}}
\put(105,60){\line(-1,1){20}}

\end{picture}\\
semiring complexity
\vspace{.1in}
\end{center}
\caption{Notions of $M$-complexity, with $M\supset\{+,*\}$.}
\label{table:nontrivial-M}
\end{table}

Now, how does $M$-complexity depend on~$M$, when there is a choice?
Here is one way to make this question precise:

\begin{problem}
\label{problem:M-m}
Let $f_1,f_2,\dots$ be a sequence of rational functions (depending on
a potentially changing set of variables) which can be computed using
the gates in $M'\subsetneq M\subset\{+,-,*,/\}$.
If the $M$-complexity of $f_n$ is $\poly(n)$, does it follow that its
$M'$-complexity is also $\poly(n)$?
\end{problem}

%To illustrate, let $M\!=\!\{*,/\}$ and $M'\!=\!\{*\}$.
%Then the answer to the question posed in Problem~\ref{problem:M-m} is
%\emph{yes}: regardless of whether division is allowed, the complexity
%of an ordinary monomial is $\poly(n)$ if and only if the log of its
%largest exponent as well as the number of the variables it involves
%are $\poly(n)$.

The nontrivial  instances of Problem~\ref{problem:M-m}, discussed
in Examples~\ref{example:strassen}--\ref{example:strassen-sf} below,
concern the four notions of $M$-complexity that involve
both additive and multiplicative operations.

\begin{example}
\label{example:strassen}
$M=\{+,-,*,/\}$, $M'=\{+,-,*\}$. In 1973, V.~Strassen \cite{Strassen}
(cf.\ \cite[Theorem 2.11]{SY})
proved that in this case,
the answer to Problem~\ref{problem:M-m} is essentially \emph{yes}: division gates
can be eliminated (at polynomial cost) provided the total degree of the polynomial $f_n$ is $\poly(n)$.

As a consequence, one for example obtains a division-free polynomial
algorithm for computing a determinant.
More efficient algorithms of this kind can be constructed directly
(ditto for the Pfaffian), see %\cite{Berkowitz}, \cite[\S 45]{Faddeev-Faddeeva}, or the survey~
\cite{Rote} and references therein.
\end{example}

\begin{example}
\label{example:valiant}
$M=\{+,-,*\}$, $M'=\{+,*\}$.
(In view of Strassen's theorem, this setting is essentially equivalent to
taking $M=\{+,-,*,/\}$, $M'=\{+,*\}$.)
In 1980, L.~Valiant~\cite{Valiant} has
shown that in this case, the answer to Problem~\ref{problem:M-m} is
\emph{no}:
for a certain sequence of polynomials $f_n$ with nonnegative integer
coefficients,
the $\{+,*\}$-complexity of $f_n$ is exponential in $n$ whereas their
$\{+,-,*\}$-complexity (equivalently, ordinary arithmetic circuit
complexity) is $\poly(n)$.
The polynomial $f_n$ used by Valiant is defined as a generating
function for perfect matchings in a particular planar graph (a
triangular grid).
By a classical result of P.~W.~Kasteleyn~\cite{Kasteleyn}, such
generating functions can be computed as certain Pfaffians, hence their
ordinary complexity is polynomial.

It is unknown whether subtraction-free complexity of Valiant's test function~$f_n$
is $\poly(n)$. If the answer is \emph{yes}, then $f_n$ exhibits a (superpolynomial) complexity gap between
subtrac\-tion-free and $\{+,*\}$-complexity. If the answer is \emph{no}, then we get a complexity gap between
ordinary and subtraction-free complexity. Thus, we have known since Valiant's work that
one of these two gaps is present in his example---but we still do not know which~one!

Other examples of polynomials $f_n$ which exhibit an exponential gap between
ordinary and $\{+,*\}$-complexity were given by M.~Jerrum and M.~Snir~\cite{Jerrum-Snir},
cf.\ Theorem~\ref{th:jerrum-snir}.

The notion of $\{+,*\}$-complexity of a polynomial with nonnegative coefficients was
already considered in 1976 by C.~Schnorr~\cite{Schnorr}.
(He used the terminology ``monotone rational computations'' which we shun.)
Schnorr gave a lower bound for $\{+,*\}$-complexity
which only depends on the \emph{support} of a polynomial,
i.e., on the set of monomials that contribute with a positive
coefficient.
Valiant's argument uses a further refinement of Schnorr's lower bound,
cf.~\cite{Shamir-Snir}.
\end{example}

\begin{example}
\label{example:sf-complexity}
$M=\{+,-,*,/\}$, $M'=\{+,*,/\}$.
In this case, Problem~\ref{problem:M-m} asks
whether any subtraction-free rational expression that can be computed
by an arithmetic circuit of polynomial size can be computed  by
such a circuit without subtraction gates.
In Section~\ref{sec:quadratic},
we show  the answer to this question to be negative,
by constructing a sequence of polynomials $f_n$ whose ordinary arithmetic circuit complexity
is $O(n)$ while their $\{+,*,/\}$-complexity is at least exponential in~$n$.
Unfortunately, this example is somewhat artificial;
it would be interesting to find an example of a natural computational problem
with an exponential gap between ordinary and subtraction-free complexity.

In the opposite direction, we demonstrate that for some important classes of functions,
the gap between these two complexity measures is merely polynomial,
in a somewhat counter-intuitive way:
these functions
turn out to have polynomial subtraction-free complexity
even though their ``naive'' sub\-trac\-tion-free description has exponential size.
%In this paper, we show that some important families of functions whose
%``naive'' sub\-trac\-tion-free description has exponential size turn out
%to have polynomial subtraction-free complexity.

Note that subtraction is the only arithmetic operation that does not
allow for an efficient control of round-up errors (for positive real
inputs).
Consequently the task of eliminating subtraction gates
is relevant to the design of numerical algorithms which are
both efficient and precise.
To rephrase, this instance of Problem~\ref{problem:M-m}
can be viewed as addressing the tradeoff between speed and accuracy.
See \cite{Demmel-Koev-MathComp} for an excellent discussion of these issues.

%Does subtraction-free complexity class of a rational function in
%Z{+,*,/} always coincide with its ordinary complexity class?
%To rephrase: Are subtraction-free computations less powerful than
%ordinary ones?

%So is there a tradeoff between speed and accuracy?
%Is there a family of functions which can be computed both quickly
%and precisely, but not at the same time?

\end{example}

\iffalse
\begin{remark}
%The following simple argument shows that in at least one of the two open problems presented in
%Examples~\ref{example:sf-complexity}--\ref{example:strassen-sf}, the answer must be negative.
Consider the sequence of generating functions $(f_n)$ used by Valiant (cf.\ Example~\ref{example:valiant}).
%Each $f_n$ is a polynomial with nonnegative coefficients,
%so the notion of $M$-complexity of $f_n$ makes sense for any $M\!\supseteq\!\{+,*\}$.
We know (see Examples~\ref{example:strassen}--\ref{example:valiant})~that
$\{+,-,*,/\}$-complexity of $f_n$ is $\poly(n)$ whereas
its $\{+,*\}$-complexity is exponential in~$n$.
Consequently, in the sequence
%the $\{+,*,/\}$-complexity is sandwiched between those two;
%consequently, at least one of the two inequalities
\begin{center}
$\{+,-,*,/\}$-complexity $<$ $\{+,*,/\}$-complexity $<$ $\{+,*\}$-complexity,
\end{center}
at least one of the two steps must give a jump from $\poly(n)$ to a superpolynomial growth rate.
%We conclude that either ``subtraction can be powerful'' (for ordinary complexity) or
%``division can be exponentially powerful within the realm of subtraction-free expressions''
%(or both, since other choices of $f_n$ might yield different outcomes).
Either of these two conclusions would be exciting to make; we just don't know which one is true.
\end{remark}
\fi

\begin{example}
\label{example:strassen-sf}
$M\!=\!\{+,*,/\}$, $M'\!=\! \{+,*\}$.
This is the subtraction-free version of the problem discussed in Example~\ref{example:strassen}.
That is: can division gates be eliminated in the absence of subtraction?
We will show that the answer is~\emph{no},
by demonstrating that the generating function for directed spanning trees
in a complete directed graph on $n$ vertices has $\poly(n)$ subtraction-free complexity.
This contrasts with an exponential lower bound for the $\{+,*\}$-complexity
of the same generating function, given by M.~Jerrum and M.~Snir~\cite{Jerrum-Snir}.
\end{example}

\section{Main results}
\label{sec:main-results}

\subsection*{Schur functions and their variations}

\emph{Schur functions} $s_\lambda(x_1,...x_k)$
(here $\lambda=(\lambda_1\ge \lambda_2\ge\cdots\ge 0)$ is an integer partition)
are remarkable symmetric polynomials that
play prominent roles in representation
theory, algebraic geometry, enumerative combinatorics, mathematical
physics, and other mathematical disciplines; see, e.g.,
\cite[Chapter~I]{Macdonald} \cite[Chapter~7]{EC2}.
Among many equivalent ways to define Schur functions
(also called \emph{Schur polynomials}), let us mention
two classical determinantal formulas: the bialternant formula and the Jacobi-Trudi formula.
These formulas are recalled in
Sections~\ref{sec:main-algorithm} and \ref{sec:skew}, respectively.

Schur functions and their numerous variations
(\emph{skew} Schur functions,
\emph{supersymmetric} Schur functions,
%\emph{factorial} Schur functions,
$Q$-~and \hbox{$P$-Schur functions}, etc., see \emph{loc.\ cit.})
provide
a natural source of computational problems whose complexity might be
sensitive to the set of allowable arithmetic operations.
On the one hand, these polynomials can be computed efficiently in an
unrestricted setting, via determinantal formulas;
on the other hand, their (nonnegative) \linebreak[3]
expansions,
as generating functions for appropriate \emph{tableaux},
are in general exponentially
long, and coefficients of individual monomials are provably hard to compute,
cf.\ Remark~\ref{rem:+*-schur}.
(Admittedly, a low-complexity polynomial can have high-complexity coefficients.
For example, the coefficient of  $x_1\cdots x_n$ in
$\prod_i \sum_j (a_{ij} x_j)$
is the permanent of the matrix $(a_{ij})$.)

The interest in determining the subtraction-free complexity of Schur functions
goes back at least as far as mid-1990s, when the problem
attracted the attention of J.~Demmel and the first author,
cf.~\cite[pp.\ 66--67]{DGESVD}.
The following result is implicit in the work of P.~Koev
\cite[Section 6]{Koev-2007}; more details can be found in \cite[Section 4]{CDEKK}).

\begin{theorem}[P.~Koev]
\label{th:schur}
Subtraction-free complexity of a Schur polynomial
$s_\lambda(x_1,...x_k)$ is at most $O(n^3)$ where $n=k+\lambda_1$.
\end{theorem}

%As usual, we denote by $|\lambda|=\sum_i \lambda_i$ the \emph{size}
%of~$\lambda$, and by $\ell(\lambda)$ the \emph{length}
%of~$\lambda$ (the number of positive parts~$\lambda_i$).

%note that
%$\ell(\lambda)\leq n$ (or else $s_\lambda(x_1,...x_n)=0$)
%and consequently $\lambda_1\leq |\lambda| \leq \ell(\lambda) \lambda_1
%\leq n\lambda_1$.

In this paper, we give an alternative proof of Theorem~\ref{th:schur}
based on the technology of \emph{cluster transformations}.
%what we view as a more conceptual and streamlined version of Koev's approach.
The algorithm presented in Section~\ref{sec:main-algorithm} computes $s_\lambda(x_1,...x_k)$
via a subtraction-free arithmetic circuit of size $O(n^3)$.
The bit complexity is $O(n^3 \log^2 n)$.

\medskip

All known fast subtraction-free algorithms for computing Schur
functions use division.
%This prompts the following question, which to the best of our
%knowledge is open.

\begin{problem}
\label{problem:+*-schur}
Is the $\{+,*\}$-complexity of a Schur function polynomial?
\end{problem}

\begin{remark}
\label{rem:+*-schur}
We suspect the answer to this question to be negative.
In any case, Problem~\ref{problem:+*-schur} is likely to be very
hard. We note that Schnorr-type lower bounds are useless in the case
of Schur functions.
Intuitively, computing a Schur function is difficult not because of
its support  but because of the complexity of its coefficients (the
Kostka numbers).
The problem of computing an individual Kostka number is known to be
\#\textbf{P}-complete (H.~Narayanan~\cite{Narayanan}) whereas the support of a Schur
function is very easy to determine.
\end{remark}

Our approach leads to the following generalizations of Theorem~\ref{th:schur}.
See Sections~\ref{sec:supersymmetric} and~\ref{sec:skew} for precise definitions as well as proofs.

\begin{theorem}
\label{th:double}
A double Schur polynomial
$s_\lambda(x_1,\dots,x_k\,|\, y)$
can be computed by a sub\-traction-free arithmetic circuit of
size $O(n^3)$ where $n=k+\lambda_1$.
%here $\lambda'_1$ denotes the first part of the conjugate.
The bit complexity of the corresponding algorithm is $O(n^3\log^2 n)$.
\end{theorem}

Theorem~\ref{th:double} can be used
to obtain an efficient subtraction-free algorithm for \emph{super\-symmetric} Schur functions,
see Theorem~\ref{eq:th:super}.
%Section~\ref{sec:supersymmetric}.

\begin{theorem}
\label{th:skew}
A skew Schur polynomial
$s_{\lambda/\nu}(x_1,\dots,x_k)$
can be computed by a subtraction-free arithmetic circuit of
size $O(n^5)$ where $n=k+\lambda_1$.
The bit complexity of the corresponding algorithm is $O(n^5\log^2 n)$.
\end{theorem}

%Theorems~\ref{th:super}--\ref{th:skew} have a common generalization, for a skew
%super-Schur polynomial $s_{\lambda/\mu}(x_1,...x_k\,|\,y_1,...,y_\ell)$.
%They also imply polynomiality of subtraction-free complexity of
%factorial Schur functions

%Our bound for Theorem~\ref{th:schur} is the same as Koev's.

\begin{remark}
The actual subtraction-free complexity
(or even the $\{+,*\}$-complexity) of a particular Schur polynomial
%$s_\lambda(x_1,\dots,x_k)$
can be significantly smaller than the upper bound
of Theorem~\ref{th:schur}.
For example, consider the bivariate Schur polynomial
$s_{(\lambda_1,\lambda_2)}(x_1,x_2)$ given by
\[
s_{(\lambda_1,\lambda_2)}(x_1,x_2)=(x_1
x_2)^{\lambda_2}h_{\lambda_1-\lambda_2}(x_1,x_2),
\]
where $h_d(x_1,x_2)=\textstyle\sum_{1\le i\le d} x_1^i\cdot
x_2^{d-i}$ (the complete homogeneous symmetric polynomial).
The polynomial $s_{(\lambda_1,\lambda_2)}(x_1,x_2)$
can be computed in $O(\log(\lambda_1))$ time using addition and multiplication only,
by iterating the formulas
\begin{align}
h_{2d+1}(x_1,x_2)&=(x_1^{d+1}+x_2^{d+1})\,h_d(x_1,x_2)\\
h_{2d+2}(x_1,x_2)&=(x_1^{d+2}+x_2^{d+2})\,h_d(x_1,x_2) + x_1^{d+1} x_2^{d+1}.
\end{align}
\end{remark}

%\begin{remark}
%It would be interesting to clarify, whether
%one can compute $s_I$ by means of a monotone computation (perhaps,
%extended by divisions) within complexity polynomial in $k,\, \log
%\deg (s_I)$ at least for small values of $k$?
%\end{remark}

\pagebreak[3]

\subsection*{Spanning trees}

We also develop efficient subtraction-free algorithms for another class
of polynomials:
the generating functions of spanning trees in weighted graphs,
either ordinary (undirected) or directed.
In the directed case, the edges of a tree should be directed towards the designated root vertex.
The weight of a tree is defined as the product of the weights of its edges.
See Sections~\ref{sec:spanning-trees}--\ref{sec:directed-spanning-trees}
for precise definitions.

Determinantal formulas for these generating functions (the Matrix-Tree Theorems)
go back to G.~Kirchhoff~\cite{Kirchhoff}
(1847, undirected case) and
W.~Tutte \cite[Theorem~6.27]{Tutte-GT} (1948, directed case).
Consequently, their ordinary complexity is polynomial.
Amazingly, the $\{+,*\}$-complexity is exponential in the directed case:

\begin{theorem}[{\rm M.~Jerrum and M.~Snir~\cite[4.5]{Jerrum-Snir}}]
\label{th:jerrum-snir}
Let $\varphi_n$ denote the generating function for directed spanning trees
in a complete directed graph on $n$ vertices.
Then the $\{+,*\}$-complexity of $\varphi_n$ can be bounded from below
by $n^{-1}(4/3)^{n-1}$.
\end{theorem}

In Sections~\ref{sec:spanning-trees}--\ref{sec:directed-spanning-trees},
we establish the following results.

\begin{theorem}
\label{th:combined-spanning-trees}
Let $G$ be a weighted simple graph (respectively, simple directed graph)
on $n$ vertices.
Then the generating function for spanning trees in~$G$
(respectively, directed spanning trees rooted at a given vertex)
can be computed by a subtraction-free
arithmetic circuit of size~$O(n^3)$.
\end{theorem}

In particular, the $\{+,*,/\}$-complexity of the
polynomials $\varphi_n$ from Theorem~\ref{th:jerrum-snir} is $O(n^3)$,
in sharp contrast with the Jerrum-Snir lower bound.

\section{Subtraction-free computation of a Schur function}
\label{sec:main-algorithm}

This section presents our proof of Theorem~\ref{th:schur},
i.e., an efficient subtraction-free algorithm for computing a
Schur function.
The basic idea of our approach is rather simple,
provided the reader is already familiar with the basics of cluster
algebras.
(Otherwise, (s)he can safely skip the next paragraph,
as we shall keep our presentation self-contained.)

A~Schur function can be given by a
determinantal formula, as a minor of a certain matrix,
and consequently can be viewed as a specialization of some
cluster variable in an appropriate cluster algebra.
It~can therefore be obtained by a sequence of subtraction-free
rational transformations (the ``cluster transformations''
corresponding to exchanges of cluster variables under cluster mutations) from
a wisely chosen initial extended cluster.
An upper bound on subtraction-free complexity is then obtained by
combining the number of mutation steps with the complexity of
computing the initial~seed.

The most naive version of this approach starts with the classical Jacobi-Trudi
formula (reproduced in Section~\ref{sec:skew}) that expresses a (more generally, skew) Schur
function as a minor of the Toeplitz matrix $(h_{i-j}(x_1,...,x_k))$
where $h_d$ denotes the $d$th complete homogeneous symmetric
polynomial, i.e., the sum of all monomials of degree~$d$.
%(Alternatively, one could use the Kostka-N\"agelsbach determinant,
%which involves elementary symmetric polynomials.)
Unfortunately, this approach (or its version employing elementary symmetric polynomials)
does not seem to yield a solution:
even though the number of mutation steps can be polynomially bounded,
we were unable to identify an initial cluster all of whose elements
are easier to compute (by a polynomial subtraction-free algorithm)
than a general Schur function.

The key idea is to employ a different cluster recurrence that iteratively computes
Schur polynomials in \emph{varying} number of arguments.
This leads us to an algorithm that ultimately relies---as did Koev's original approach~\cite{Koev-2007}---on
another classical determinantal formula for a Schur function, which goes back to Cauchy and Jacobi.
This formula expresses $s_\lambda$ as a ratio of two ``alternants,'' i.e.,
Vandermonde-like determinants. Let us recall this formula in the form
that will be convenient for our purposes;
an~uninitiated reader can view it as a \emph{definition} of a Schur function.

%\begin{definition}
Let $n$ be a positive integer.
Consider the $n\times n$ ``rescaled Vandermonde'' matrix
%defined by dividing the $j$th column of the Vandermonde matrix $(x_j^{i-1})$
%by $(x_j-x_1)\cdots(x_j-x_{j-1})$:
\begin{equation}
\label{eq:Xij}
X =(X_{ij})
=\left(\frac{x_j^{i-1}}{\displaystyle\prod_{a<j}(x_j-x_a)}
\right)_{\!\!\! i,j=1}^{\!\!\!n}
%_{\!\!\!\!\substack{1\le i\le n\\ 1\le j\le k}}
=\left(
\begin{matrix}
1 & \dfrac{1}{x_2-x_1} & \dfrac{1}{(x_3-x_1)(x_3-x_2)} &\cdots\\[.2in]
x_1 & \dfrac{x_2}{x_2-x_1} & \dfrac{x_3}{(x_3-x_1)(x_3-x_2)} &\cdots\\[.2in]
x_1^2 & \dfrac{x_2^2}{x_2-x_1} & \dfrac{x_3^2}{(x_3-x_1)(x_3-x_2)}
&\cdots\\ %[.1in]
%x_1^3 & \dfrac{x_2^3}{x_2-x_1} & \dfrac{x_3^3}{(x_3-x_1)(x_3-x_2)}
\vdots&\vdots&\vdots&\ddots
\end{matrix}
\right)\!.
\end{equation}
For a subset $I\subset %[n]\stackrel{\rm def}{=}
\{1,\dots,n\}$, say of cardinality~$k$,
let $s_I$ denote the corresponding ``flag minor'' of~$X$,
i.e., the determinant of the square submatrix of $X$ formed by the intersections of the rows in~$I$
and the first $k$ columns:
\begin{equation}
\label{eq:s_I}
s_I=s_I(x_1,\dots,x_k)=\det(X_{ij})_{i\in I, j\le k} .
\end{equation}
(For example, $s_{1,\dots,k}=\det(X_{ij})_{i,j=1}^{k}=1$.)
It is easy to see that $s_I$ is a symmetric polynomial in the
variables $x_1,\dots,x_k$.

Now, let $\lambda=(\lambda_1,\dots,\lambda_k)$ be a partition with at most $k$ parts
satisfying $\lambda_1+k\le n$.
Define the $k$-element subset $I(\lambda)\subset\{1,\dots,n\}$ by
\begin{equation}
\label{eq:I(lambda)}
I(\lambda)=\{\lambda_k+1, \lambda_{k-1}+2,\ldots, \lambda_1+k\}.
\end{equation}
The Schur function/polynomial $s_\lambda(x_1,\dots,x_k)$ is then given by
\begin{equation}
\label{eq:s_lambda}
s_\lambda(x_1,\dots,x_k)
=s_{I(\lambda)}(x_1,\dots,x_k)
=\det(X_{ij})_{i\in I(\lambda), j\le k}.
\end{equation}
If $\lambda$ has more than $k$ parts, then $s_\lambda(x_1,\dots,x_k)\!=\!0$.

We note that as $I$ ranges over all subsets of $\{1,\dots,n\}$,
the flag minors of $X$ range over the nonzero Schur polynomials
$s_\lambda(x_1,\dots,x_k)$ with $\lambda_1+k\le n$.

\medskip

Flag minors play a key role in one of the most
important examples of cluster algebras, the coordinate ring of
the base affine space. % see \cite{icm} and references therein.
Let us briefly recall (borrowing heavily from~\cite{icm}, and glossing
over technical details, which can be found in \emph{loc.\ cit.})
the basic features of the underlying
combinatorial setup, which was first introduced in~\cite{BFZ};
cf.\ also~\cite{DKKumn}.

A \emph{pseudoline arrangement} is a collection of $n$ curves
(``pseudolines'')
each of which is a graph of a continuous function on $[-1,1]$;
each pair of pseudolines must have exactly one crossing point in
common;
no three pseudolines may intersect at a point.
See Figure~\ref{fig:4-pseudo}.
The pseudolines are numbered $\mathbf{1}$~through~$\mathbf{n}$ from
the bottom up along the left border.
The resulting  pseudoline arrangement is considered up to isotopy
(performed within the space of such arrangements).

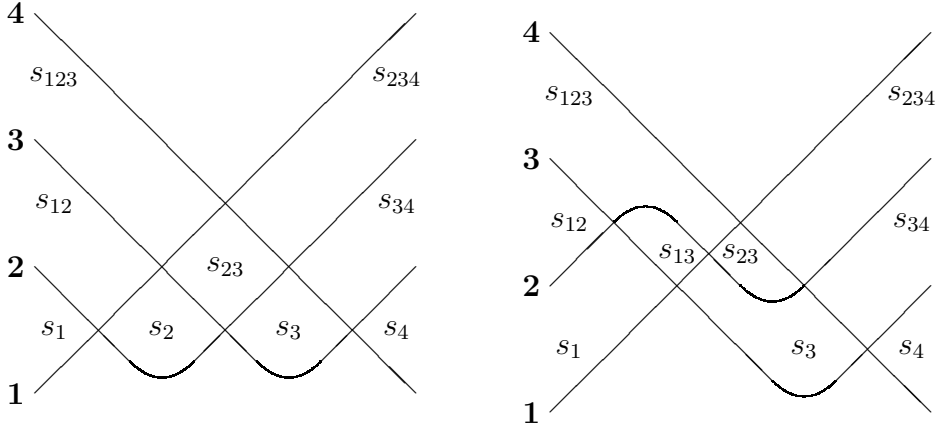
\begin{figure}[ht]
\begin{center}
\setlength{\unitlength}{2.4pt}
\begin{picture}(60,60)(0,0)
\thinlines

\put(0,0){\line(1,1){60}}
\put(0,20){\line(1,-1){15}}
\put(0,40){\line(1,-1){35}}
\put(60,40){\line(-1,-1){35}}
\put(60,20){\line(-1,-1){15}}
\put(0,60){\line(1,-1){60}}

\qbezier(35,5)(40,0)(45,5)
\qbezier(15,5)(20,0)(25,5)
%\put(10,10){\circle*{4}}

\put(-3,0){\makebox(0,0){$\mathbf{1}$}}
\put(-3,20){\makebox(0,0){$\mathbf{2}$}}
\put(-3,40){\makebox(0,0){$\mathbf{3}$}}
\put(-3,60){\makebox(0,0){$\mathbf{4}$}}

\put( 3,10){\makebox(0,0){$s_{1}$}}
\put(20,10){\makebox(0,0){$s_{2}$}}
\put(40,10){\makebox(0,0){$s_{3}$}}
\put(57,10){\makebox(0,0){$s_{4}$}}

\put(3,30){\makebox(0,0){$s_{12}$}}
\put(30,20){\makebox(0,0){$s_{23}$}}
\put(57,30){\makebox(0,0){$s_{34}$}}

\put(3,50){\makebox(0,0){$s_{123}$}}
\put(57,50){\makebox(0,0){$s_{234}$}}

\end{picture}
\qquad\qquad
\begin{picture}(60,60)(0,3)
\thinlines

\put(0,0){\line(1,1){60}}
%\put(0,20){\line(1,-1){5}}
\put(0,40){\line(1,-1){35}}
\put(60,40){\line(-1,-1){20}}
\put(60,20){\line(-1,-1){15}}
\put(0,60){\line(1,-1){60}}
\put(20,30){\line(1,-1){10}}

\qbezier(30,20)(35,15)(40,20)

\qbezier(35,5)(40,0)(45,5)

\qbezier(10,30)(15,35)(20,30)
\put(0,20){\line(1,1){10}}
%\put(10,10){\circle*{4}}

\put(-3,0){\makebox(0,0){$\mathbf{1}$}}
\put(-3,20){\makebox(0,0){$\mathbf{2}$}}
\put(-3,40){\makebox(0,0){$\mathbf{3}$}}
\put(-3,60){\makebox(0,0){$\mathbf{4}$}}

\put( 3,10){\makebox(0,0){$s_{1}$}}
\put(20,25){\makebox(0,0){$s_{13}$}}
\put(40,10){\makebox(0,0){$s_{3}$}}
\put(57,10){\makebox(0,0){$s_{4}$}}

\put(3,30){\makebox(0,0){$s_{12}$}}
\put(30,25){\makebox(0,0){$s_{23}$}}
\put(57,30){\makebox(0,0){$s_{34}$}}

\put(3,50){\makebox(0,0){$s_{123}$}}
\put(57,50){\makebox(0,0){$s_{234}$}}

\end{picture}

\end{center}
\caption{Two pseudoline arrangements, and associated chamber minors}
\label{fig:4-pseudo}
\end{figure}

To each region $R$ of a pseudoline arrangement,
except for the very top and the very bottom,
we associate the flag minor $s_{I(R)}$ %of an $n\times n$ matrix $X=(X_{ij})$
%(cf.~\cite{BFZ})
indexed by the set $I(R)$ of labels
of the pseudolines passing below~$R$.
These are called \emph{chamber minors}.

Pseudoline arrangements are related to each other via sequences of
\emph{local moves} of the form shown in Figure~\ref{fig:move}.
Each local move results in replacing exactly one chamber minor
$s_{I(R)}$ by a new one;
these two minors are denoted by $e$ and $f$ in Figure~\ref{fig:move}.
To illustrate,~the two pseudoline arrangements in
Figure~\ref{fig:4-pseudo} are related by a local move that replaces
$s_{2}$ by~$s_{13}$ (or vice versa).

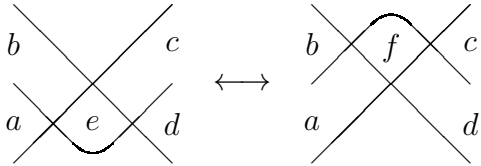
\begin{figure}[ht]
\begin{center}
\setlength{\unitlength}{1.5pt}
\begin{picture}(40,40)(0,3)
\thinlines
\put(0,0){\line(1,1){40}}
\put(0,40){\line(1,-1){40}}
\put(0,20){\line(1,-1){15}}
\put(40,40){\line(-1,-1){35}}
\put(40,20){\line(-1,-1){15}}
\qbezier(15,5)(20,0)(25,5)
\put(0,10){\makebox(0,0){$a$}}
\put(0,30){\makebox(0,0){$b$}}
\put(40,30){\makebox(0,0){$c$}}
\put(40,10){\makebox(0,0){$d$}}
\put(20,10){\makebox(0,0){$e$}}
\end{picture}
\begin{picture}(30,40)(0,3)
\put(15,20){\makebox(0,0){$\longleftrightarrow$}}
\end{picture}
\begin{picture}(40,40)(0,3)
\thinlines
\put(0,0){\line(1,1){40}}
\put(0,40){\line(1,-1){40}}
\put(0,20){\line(1,1){15}}
\put(40,40){\line(-1,-1){35}}
\put(40,20){\line(-1,1){15}}
\qbezier(15,35)(20,40)(25,35)
\put(0,10){\makebox(0,0){$a$}}
\put(0,30){\makebox(0,0){$b$}}
\put(40,30){\makebox(0,0){$c$}}
\put(40,10){\makebox(0,0){$d$}}
\put(20,29){\makebox(0,0){$f$}}
\end{picture}
\end{center}
\caption{A local move in a pseudoline arrangement}
\label{fig:move}
\end{figure}

The key observation made in~\cite{BFZ} is that
the chamber minors $a,b,c,d,e,f$ associated with the regions
surrounding the local
move (cf.\ Figure~\ref{fig:move}) satisfy the identity
\begin{equation}
\label{eq:ef=ac+bd}
ef=ac+bd.
\end{equation}
Thus $f$ can be written as a subtraction-free expression in $a,b,c,d,e$,
and similarly $e$ in terms of $a,b,c,d,f$.
It is not hard to see that any flag minor $s_I$ appears as a chamber minor
in \emph{some} pseudo\-line
arrangement (we elaborate on this point later in this section).
Consequently, by iterating the birational transformations associated with local moves,
one can get $s_I$ as a subtraction-free rational expression in the
chamber minors of any particular initial arrangement.

To complete the proof of Theorem~\ref{th:schur},
i.e., to design a subtraction-free algorithm computing a Schur
polynomial~$s_I$ in $O(n^3)$ steps,
we need to identify an initial pseudoline
arrangement (an ``initial seed'' in cluster algebras lingo)
such that
\begin{align}
\label{eq:strategy-a}
&\text{the chamber minors for the initial seed can be computed by a subtraction-free}\\[-5pt]
\notag
&\text{arithmetic circuit of size $O(n^3)$, and}\\
\label{eq:strategy-b}
&\text{for any subset $I\subset \{1,\dots,n\}$,
the initial pseudoline arrangement can be trans-}\\[-5pt]
\notag
&\text{formed into one containing $s_I$ among its chamber minors by $O(n^3)$ local moves.}
\end{align}

\begin{remark}
\label{rem:bit-complexity-arrangements}
At this point, some discussion of bit complexity is in order.
Readers not interested in this issue may skip this Remark.

Each local move ``flips'' a triangle formed by some triple of
pseudolines with labels $i<j<k$.
%the move consists in dragging one of the pseudolines through the
%intersection of two others.
(To illustrate, the arrangements in Figure~\ref{fig:4-pseudo}
are related by the local move labeled by the triple $(1,2,3)$.)
A~sequence of say $N$ local moves (cf.~\eqref{eq:strategy-b}) can be encoded by the
corresponding sequence of triples
\begin{equation}
\label{eq:sequence-of-triples}
(i_1,j_1,k_1),\dots,(i_N,j_N,k_N).
\end{equation}
The bit complexity of our algorithm will be obtained by adding the
following contributions:
\begin{itemize}
\item
the bit complexity of computing the initial chamber minors;
\item
the bit complexity of generating the sequence of triples
\eqref{eq:sequence-of-triples};
\item
the bit complexity of performing the corresponding local moves.
\end{itemize}
Concerning the last item, note that in order to execute each of the $N$ local moves,
we will need to determine which arithmetic operations to perform
(there will be O(1) of them),
and how to transform the data structure that encodes the
pseudoline arrangement at hand, so as to reflect the changing
combinatorics of the arrangement.
The data structure that we suggest to use
is a graph $G$ on $\binom{n}{2}+2n$ vertices which include
the vertices $v_{ij}$ representing pairwise intersections of pseudolines,
together with the vertices $v_i^\text{left}$ and $v_i^\text{right}$ representing
their left and right endpoints.
At each vertex $v$ in $G$, we store the following information:
\begin{itemize}
\item
for each pseudoline passing through~$v$, the vertex (if any) that
immediately precedes~$v$ on that pseudoline, and also the vertex that
immediately follows~$v$;
\item
the set~$I$ labelling the chamber directly underneath~$v$; and
\item
the corresponding chamber minor~$s_I$.
\end{itemize}
With this in place, the local move labeled by a triple $(i,j,k)$
is performed by identifying the (pairwise adjacent) vertices of $G$ lying at the
intersections of the pseudolines with labels $i,j,k$,
changing the local combinatorics of the graph~$G$ in the vicinity of this triangle,
and performing the appropriate subtraction-free computation.
For each of the $N$ local moves,
the number of macroscopic operations involved is $O(1)$,
so the bit complexity of each move is polynomial in the size of the numbers involved
(which is going to be logarithmic in~$n$).
\end{remark}

We proceed with the design of an efficient subtraction-free algorithm for
computing a Schur polynomial~$s_I$, following the approach outlined in
\eqref{eq:strategy-a}--\eqref{eq:strategy-b}.

Our choice of the initial arrangement is the ``special'' pseudoline arrangement $A^\circ$
shown in Figure~\ref{fig:special-pseudo}
(cf.\ also Figure~\ref{fig:4-pseudo} on the~left).

\begin{figure}[ht]
\begin{center}
\setlength{\unitlength}{2.4pt}
\begin{picture}(60,60)(0,0)
\thinlines

\put(0,10){\line(1,-1){5}}
\put(0,20){\line(1,-1){15}}
\put(0,30){\line(1,-1){25}}
\put(0,40){\line(1,-1){35}}
\put(0,50){\line(1,-1){45}}
\put(0,60){\line(1,-1){60}}

\put(60,60){\line(-1,-1){60}}
\put(60,50){\line(-1,-1){45}}
\put(60,40){\line(-1,-1){35}}
\put(60,30){\line(-1,-1){25}}
\put(60,20){\line(-1,-1){15}}
\put(60,10){\line(-1,-1){5}}

\qbezier(5,5)(10,0)(15,5)
\qbezier(15,5)(20,0)(25,5)
\qbezier(25,5)(30,0)(35,5)
\qbezier(35,5)(40,0)(45,5)
\qbezier(45,5)(50,0)(55,5)
%\put(10,10){\circle*{4}}

\put(-3,0){\makebox(0,0){$\mathbf{1}$}}
\put(-3,10){\makebox(0,0){$\mathbf{2}$}}
\put(-3,20){\makebox(0,0){$\mathbf{3}$}}
\put(-3,31){\makebox(0,0){$\mathbf{\vdots}$}}
\put(-3,41){\makebox(0,0){$\mathbf{\vdots}$}}
\put(-6.5,50){\makebox(0,0){$\mathbf{n\!-\!1}$}}
\put(-3,60){\makebox(0,0){$\mathbf{n}$}}

\end{picture}
\end{center}
\caption{The special pseudoline arrangement $A^\circ$.}
\label{fig:special-pseudo}
\end{figure}
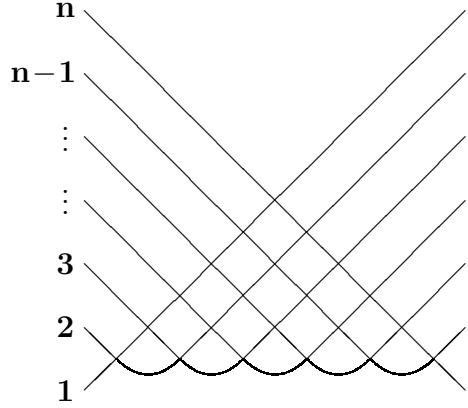

The special arrangement $A^\circ$ works well for our purposes,
for the following reason. The $\frac{n(n+1)}{2}-1$
chamber minors $s_I$ for $A^\circ$ are
indexed by the intervals
\begin{equation}
\label{eq:interval}
I%=[\ell,\ell+k-1]
=\{\ell,\ell+1,\dots,\ell+k-1\}\subsetneq \{1,\dots,n\}.
\end{equation}
%for $1\le \ell\le m\le n$, $[\ell,m]\neq [1,n]$.
Moreover such a flag minor $s_I$ is nothing but the monomial
$(x_1\cdots x_k)^{\ell-1}$:
\begin{equation}
\label{eq:special-minor}
s_I=\det\!\left(\frac{x_j^{i-1}}{\displaystyle\prod_{a<j}(x_j-x_a)}
\right)_{\!\!\!\substack{\ell\le i\le \ell+k-1\\
1\le j\le k}}
=(x_1\cdots x_k)^{\ell-1}
\,\frac{\det(x_j^{i-1})_{i,j=1}^k}{\displaystyle\prod_{a<j\le k}(x_j-x_a)}
=(x_1\cdots x_k)^{\ell-1}.
\end{equation}
(This can also be easily seen using the combinatorial definition of a
Schur function in terms of Young tableaux.)
The collection of monomials $(x_1\cdots x_k)^{\ell-1}$ can be computed
using $O(n^2)$ multiplications,
so condition \eqref{eq:strategy-a} is satisfied.

To satisfy condition~\eqref{eq:strategy-b}, at least two alternative
strategies can be used, described below under the headings Plan~A and Plan~B.

\medskip
\pagebreak[3]

\noindent
\textbf{Plan~A: Combinatorial deformation.}
The pseudocode given below  in \eqref{eq:pseudocode-moves}
produces a sequence of $O(n^3)$ local moves
transforming the special arrangement $A^\circ$
into a particular pseudoline arrangement $A_I$ containing $s_I$ as a chamber minor:
\begin{equation}
\label{eq:pseudocode-moves}
\begin{array}{l}
\text{\textbf{for} $k:=n$ \textbf{downto} 3 \textbf{do}}\\
\quad \text{\textbf{if} $k\in I$ \textbf{then}  \textbf{for} $j:=k-1$ \textbf{downto} 2  \textbf{do}}\\
\qquad \qquad \qquad\qquad\  \text{\textbf{for} $i:=j-1$ \textbf{downto} 1  \textbf{do} $\operatorname{flip}(i,j,k)$}\\
\end{array}
\end{equation}
Figure~\ref{fig:combin-deform} illustrates the above algorithm.
Its rather straightforward justification is omitted.

\begin{figure}[ht]
\begin{center}
\setlength{\unitlength}{2.4pt}
\begin{picture}(60,70)(-5,-10)
\thinlines

\put(0,10){\line(1,-1){5}}
\put(0,20){\line(1,-1){15}}
\put(0,30){\line(1,-1){25}}
\put(0,40){\line(1,-1){35}}
%\put(0,50){\line(1,-1){45}}
\put(0,60){\line(1,-1){60}}

\put(60,60){\line(-1,-1){60}}
\put(60,50){\line(-1,-1){45}}
\put(60,40){\line(-1,-1){35}}
\put(60,30){\line(-1,-1){25}}
\put(60,20){\line(-1,-1){15}}
\put(60,10){\line(-1,-1){20}}

\put(0,0){\line(-1,0){20}}
\put(0,10){\line(-1,0){20}}
\put(0,20){\line(-1,0){20}}
\put(0,30){\line(-1,0){20}}
\put(0,40){\line(-1,0){20}}
%\put(-15,50){\line(-1,0){5}}
\put(0,60){\line(-1,0){20}}

\put(-20,50){\line(1,-4){15}}
\put(-5,-10){\line(1,0){45}}

\qbezier(5,5)(10,0)(15,5)
\qbezier(15,5)(20,0)(25,5)
\qbezier(25,5)(30,0)(35,5)
\qbezier(35,5)(40,0)(45,5)
%\qbezier(45,5)(50,0)(55,5)

\put(-23,0){\makebox(0,0){$\mathbf{1}$}}
\put(-23,10){\makebox(0,0){$\mathbf{2}$}}
\put(-23,20){\makebox(0,0){$\mathbf{3}$}}
\put(-23,30){\makebox(0,0){$\mathbf{4}$}}
\put(-23,40){\makebox(0,0){$\mathbf{5}$}}
\put(-23,50){\makebox(0,0){$\mathbf{6}$}}
\put(-23,60){\makebox(0,0){$\mathbf{7}$}}

\put(20,-5){\makebox(0,0){$s_{6}$}}

\end{picture}
\qquad\qquad
\begin{picture}(60,70)(-20,-10)
\thinlines

\put(0,10){\line(1,-1){5}}
\put(0,20){\line(1,-1){15}}
%\put(0,30){\line(1,-1){25}}
\put(0,40){\line(1,-1){35}}
%\put(0,50){\line(1,-1){45}}
\put(0,60){\line(1,-1){60}}

\put(60,60){\line(-1,-1){60}}
\put(60,50){\line(-1,-1){45}}
\put(60,40){\line(-1,-1){35}}
\put(60,30){\line(-1,-1){35}}
\put(60,20){\line(-1,-1){15}}
\put(60,10){\line(-1,-1){20}}

\put(0,0){\line(-1,0){20}}
\put(0,10){\line(-1,0){20}}
\put(0,20){\line(-1,0){20}}
\put(-10,30){\line(-1,0){10}}
\put(0,40){\line(-1,0){20}}
%\put(-15,50){\line(-1,0){5}}
\put(0,60){\line(-1,0){20}}

\put(-20,50){\line(1,-4){15}}
\put(-5,-10){\line(1,0){45}}

\put(-10,30){\line(1,-4){8.75}}
\put(-1.25,-5){\line(1,0){26.25}}

\qbezier(5,5)(10,0)(15,5)
\qbezier(15,5)(20,0)(25,5)
%\qbezier(25,5)(30,0)(35,5)
\qbezier(35,5)(40,0)(45,5)
%\qbezier(45,5)(50,0)(55,5)

\put(-23,0){\makebox(0,0){$\mathbf{1}$}}
\put(-23,10){\makebox(0,0){$\mathbf{2}$}}
\put(-23,20){\makebox(0,0){$\mathbf{3}$}}
\put(-23,30){\makebox(0,0){$\mathbf{4}$}}
\put(-23,40){\makebox(0,0){$\mathbf{5}$}}
\put(-23,50){\makebox(0,0){$\mathbf{6}$}}
\put(-23,60){\makebox(0,0){$\mathbf{7}$}}

\put(12.5,-1.5){\makebox(0,0){$s_{46}$}}

\end{picture}
\end{center}
\caption{Executing Algorithm~\eqref{eq:pseudocode-moves} for $n=7$.
Shown are the pseudoline arrangements $A_I$ for $I\!=\!\{6\}$ (on the left) and $I\!=\!\{4,6\}$ (on the right).
Algorithm~\eqref{eq:pseudocode-moves} deforms $A^\circ$
(cf.\ Figure~\ref{fig:special-pseudo}) into~$A_{\{4\}}$
and then into~$A_{\{4,6\}}$.}
\label{fig:combin-deform}
\end{figure}
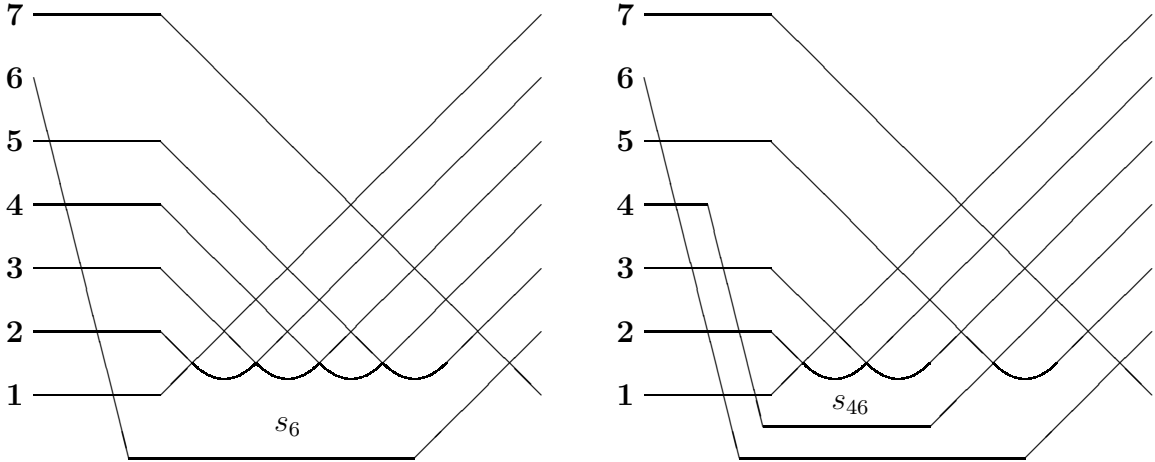

\medskip

\noindent
\textbf{Plan~B: Geometric deformation.}
Here we present an alternative solution of a more geometric flavor.
The basic idea is rather simple.
Fix a nonempty subset $I\subsetneq \{1,\dots,n\}$.
Suppose that we are able to build an arrangement $A_I$
such that
\begin{itemize}
\item
$A_I$ consists of straight line segments $L_i$; % with right endpoints $(1,-i)$;
%or more precisely graphs of linear functions on $[-1,1]$;
\item
one of the chamber minors of $A_I$ is $s_I$;
\item
$A_I$ is a ``sufficiently generic'' arrangement with these properties.
\end{itemize}
The special arrangement $A^\circ$ can be easily realized using straight segments.
We then continuously deform $A^\circ$ into~$A_I$ in the following way.
As the time parameter $t$ changes from $0$ to~$1$,
each line segment $L_i(t)$ is going to change from $L_i(0)=L_i^\circ$
to $L_i(1)=L_i$ so that each endpoint of $L_i(t)$ moves at constant speed.
It is possible to show that in the process of such deformation,
the triangle formed by each triple of lines
gets ``flipped'' at most once.
We thus obtain a sequence of at most $\binom{n}{3}$ local moves
transforming $A^\circ$ into~$A_I$, as desired.

The rest of this section is devoted to filling in the gaps left over
in the above outline.
This can be done in many different ways; the specific implementation presented below
was chosen for purely technical reasons.
We assume throughout that $n\ge 3$.

First, we realize $A^\circ$ by the collection of straight line segments
$L^\circ_1,\dots, L^\circ_n$ where $L^\circ_i$
connects the points $(-1,i^2)$ and $(1,-i)$.
Calculations show that the segments $L^\circ_i$ and $L^\circ_j$
intersect at a point $(u^\circ_{ij},v^\circ_{ij})$ with $u^\circ_{ij}=1-\frac{2}{i+j+1}$.
Consequently, for any $i<j<k$ we have $u^\circ_{ij}<u^\circ_{ik}<u^\circ_{jk}$,
implying that the arrangement's topology is as shown in
Figure~\ref{fig:special-pseudo}.

We next construct the arrangement $A_I$.
It consists of the line segments~$L_1,\dots,L_n$
such that $L_i$ has the endpoints $(1,-i)$
and $(-1,i-2\sigma _i \aa-2i^3 \aa^2)$
where
\begin{align*}
\aa&=n^{-6},\\
\sigma_i&=\begin{cases}
0 & \text{if $i\in I$;} \\
-1 & \text{if $i\notin I$.}
\end{cases}
\end{align*}
Thus $L_i$ is a segment of the straight line given by the equation
\[
y=-ix + (\sigma_i \aa +i^3\aa^2)(x-1).
\]
It is easy to see that the left (respectively right) endpoints of $L_1,\dots,L_n$
are ordered bottom-up (respectively top-down).
Consequently each pair  $(L_i,L_j)$ intersects at
a point $(x,y)$ with $-1\le x\le 1$.
Moreover one can check that all these crossing
points are distinct.
Most importantly, $L_i$ contains the point $(0,-\sigma_i \aa-i^3 \aa^2)$,
so the origin $(0,0)$ lies above $L_i$ if and only if $i\in I$;
thus the corresponding chamber minor is~$s_I$.

Let us now examine the deformation of $A^\circ$ into $A_I$ that we
described above.
As $t$ varies from $0$ to~$1$,
the right endpoint of the $i$th line segment $L_i(t)$ remains fixed at $(1,-i)$,
while the left endpoint moves at constant speed from its initial location at $(-1,i^2)$
to the corresponding location for~$A_I$.
Specifically, the left endpoint of $L_i(t)$ is
$(-1,b_i(t))$ where
\begin{equation}
\label{eq:b_i(t)}
b_i(t)=i^2-t (2\sigma_i \aa + 2 i^3\aa^2 - i+i^2).
\end{equation}
%The equation of the line containing $L_i(t)$ is thus
%\[
%y=-ix+(i-i^2+t(2\sigma_i \aa + 2 i^3\aa^2 - i+i^2))(x-1)/2.
%\]
The ordering of the endpoints remains intact:
$b_1(t)<\cdots<b_n(t)$ for $0\le t\le 1$.
Thus the intervals $L_i(t)$ form a (pseudo)line arrangement
unless some three of them are concurrent.

\begin{lemma}
\label{lem:critical}
At any time instant $0\le t\le 1$, no four intervals $L_i(t)$, have a common point.
\end{lemma}

\begin{proof}
Let $t$ be such that distinct segments
$L_i(t)$, $L_j(t)$, and $L_k(t)$ have a common point.
Then we have the identity
\begin{equation}
\label{eq:critical-t-ijk}
(b_i(t)-b_j(t)) (i-j)^{-1}-(b_i(t)-b_k(t)) (i-k)^{-1}=0.
\end{equation}
Substituting \eqref{eq:b_i(t)} into \eqref{eq:critical-t-ijk} and dividing by $j-k$, we obtain
\begin{equation}
\label{eq:equation-for-t-ijk}
1-t+2\aa t \frac{\sigma_i(k-j)+\sigma_j(i-k)+\sigma_k(j-i)}{(i-j)(i-k)(j-k)} - 2\aa^2 t (i+j+k)=0.
\end{equation}
The (unique) time instant $t=t_{ijk}$ at which $L_i(t)$, $L_j(t)$, and $L_k(t)$
are concurrent can be found from the linear equation~\eqref{eq:equation-for-t-ijk}.
(If the solution does not satisfy $t_{ijk}\in[0,1]$, then such a time instant does not exist.)

Now suppose that $j'\notin\{i,j,k\}$ is such that $L_i(t)$, $L_{j'}(t)$, and $L_k(t)$
are concurrent at the same moment $t=t_{ijk}$.
Then \eqref{eq:equation-for-t-ijk} holds with $j$ replaced by~$j'$.
Subtracting one equation from the other and dividing by $2\aa t$, we obtain:
\begin{equation}
\label{eq:j-j'}
\frac{\sigma_i(k-j)+\sigma_j(i-k)+\sigma_k(j-i)}{(i-j)(i-k)(j-k)}
-\frac{\sigma_i(k-j')+\sigma_{j'}(i-k)+\sigma_k(j'-i)}{(i-j')(i-k)(j'-k)}
=\aa(j-j').
\end{equation}
This yields the desired contradiction. Indeed, the right-hand side of \eqref{eq:j-j'}
is nonzero, and less than~$n^{-5}$ in absolute value,
whereas the left-hand side, if nonzero, is a rational number with
denominator at most~$n^5$.
\iffalse
Hence $$t^{-1}=1+2\cdot ((\sigma_i-2\cdot \sigma_j+\sigma_k)\cdot
a\, +\, (i-2\cdot j +k)\cdot a^2 \, +\, (i^3-2\cdot j^3 +k^3)\cdot
a^3)\cdot ((i-j)\cdot (k-j)\cdot (i-k))^{-1},$$ \noindent provided
that $0\le t\le 1$. Treating the right-hand side of the latter
equality as a polynomial $p_{ijk}(a)$ in $a$, a direct calculation
shows that for another triple $i>j_1>k$ we have that polynomial
$p_{ijk}(a)-p_{ij_1k}(a)$ (being, in fact, either quadratic or
linear after dividing by $a$) does not vanish identically in $a$,
and the chosen value $a=n^{-4}$ is not its root. The latter follows
from the remark that a root $c,\, |c|<1$ of a polynomial $\sum_{i\ge
i_0} q_i\cdot x^i$ fulfils inequality $|c|>(1+b)^{-1}$ where
$b:=\max _{i>i_0} |q_i/q_{i_0}|$.
\fi
\end{proof}

In view of Lemma~\ref{lem:critical}, at each time instant~$t=t_{ijk}\in [0,1]$
satisfying equation~\eqref{eq:equation-for-t-ijk} for some triple of
distinct indices $i,j,k\in\{1,\dots,n\}$,
our pseudoline arrangement undergoes (potentially several, commuting with each other)
local moves associated with the corresponding triple intersections
of line segments $L_i(t), L_j(t), L_k(t)$.
%Such a move replaces the chamber minor
%$s_{J\cup\{j\}}$ by $s_{J\cup{i,k}}$ for a suitable subset
%$J\subset [n]$.

Our algorithm computes the numbers $t_{ijk}$ via~\eqref{eq:equation-for-t-ijk},
selects those satisfying \hbox{$0\le t_{ijk}\le 1$,}
and orders them in a non-decreasing order.
This yields a sequence of $O(n^3)$ local moves
transforming $A^\circ$ into~$A_I$.
To estimate the bit complexity, we refer to
Remark~\ref{rem:bit-complexity-arrangements},
and note that the bit size of %each rational number
$t_{ijk}$ is bounded by $O(\log n)$.
The algorithm invokes a sorting algorithm \cite{Aho} to order
$O(n^3)$ numbers $t_{ijk}$, so its bit complexity is bounded
by $O(n^3\cdot \log ^2 n)$.

\iffalse
\begin{theorem}
\label{complexity}
%For any set $\emptyset \neq I\subsetneq [n]$ one can construct a
%sequence of local moves leading from a certain cluster
%containing minor $X_I$ to the standard cluster. The length of this
%sequence does not exceed $O(n^3)$, and the complexity of this
%construction is
The algorithm described above carries out a subtraction-free computation
of the Schur polynomial~$s_I$, for any $I\subset [n]$.
The bit complexity of this algorithm is
bounded from above by $O(n^3 \cdot \log ^2 n)$.
\end{theorem}
\fi

\begin{remark}
Our algorithm demonstrates that the positivity of the coefficients of a Schur polynomial
(as defined by the ``bialternant  formula'' \eqref{eq:s_lambda})
can be viewed as an instance of positivity of Laurent expansions of cluster variables,
a general property that conjecturally holds in any cluster algebra, see \cite[p.~499]{FZ}.
\end{remark}

\section{Double and supersymmetric Schur functions}
\label{sec:supersymmetric}

In this section, we present efficient subtraction-free algorithms for computing
double and supersymmetric Schur
polynomials. % $s_\lambda(x_1,\dots,x_k\,|\,y_1,\dots,y_\ell)$.
These polynomials play important role in representation theory and other areas
of mathematics, see, e.g., \cite{Goulden-Greene, Macdonald-92, Molev}
and references therein.
Our notational conventions are close to those in \cite[6th~Variation]{Macdonald-92};
the latter conventions differ from some other literature including~\cite{Molev}.

\subsection*{Double Schur functions}

Let $y_1,y_2,\dots$ be a sequence of formal variables.
Double Schur functions $s_\lambda(x_1,\dots,x_k|y)$
are generalizations of ordinary Schur functions
$s_\lambda(x_1,\dots,x_k)$
which depend on additional parameters~$y_i$.
The definition given below is a direct generalization of the
definition of $s_\lambda(x_1,\dots,x_k)$ given in
Section~\ref{sec:main-algorithm}.

Let $Z=(Z_{ij})_{i,j=1}^n$ be the $n\times n$ matrix defined by
\begin{equation}
\label{eq:Zij}
Z_{ij}=\frac{\displaystyle\prod_{1\le
    b<i}(x_j+y_b)}{\displaystyle\prod_{1\le a<j}(x_j-x_a)},
\end{equation}
cf.\ \eqref{eq:Xij}.
Thus
\[
Z=(Z_{ij})
=\left(
\begin{matrix}
1 & \dfrac{1}{x_2-x_1} & \dfrac{1}{(x_3-x_1)(x_3-x_2)} &\cdots\\[.2in]
x_1+y_1 & \dfrac{x_2+y_1}{x_2-x_1} & \dfrac{x_3+y_1}{(x_3-x_1)(x_3-x_2)} &\cdots\\[.2in]
(x_1+y_1)(x_1+y_2) & \dfrac{(x_2+y_1)(x_2+y_2)}{x_2-x_1} &
\dfrac{(x_3+y_1)(x_3+y_2)}{(x_3-x_1)(x_3-x_2)}
&\cdots\\ %[.1in]
%x_1^3 & \dfrac{x_2^3}{x_2-x_1} & \dfrac{x_3^3}{(x_3-x_1)(x_3-x_2)}
\vdots&\vdots&\vdots&\ddots
\end{matrix}
\right)\!.
\]
For $I\subset \{1,\dots,n\}$ of cardinality~$k$, we set (cf.~\eqref{eq:s_I})
\begin{equation}
\label{eq:s_I(x|y)}
s_I(x_1,\dots,x_k|y)=\det(Z_{ij})_{i\in I, j\le k} .
\end{equation}
As before, $s_I(x_1,\dots,x_k|y)$ is a symmetric polynomial in
$x_1,\dots,x_k$.

Now let $\lambda=(\lambda_1,\dots,\lambda_k)$ be
a partition with at most $k$ parts satisfying $\lambda_1+k\le n$.
The~\emph{double Schur polynomial} $s_\lambda(x_1,\dots,x_k|y)$
is the polynomial in the variables $x_1,\dots,x_k$ and
$y_1,\dots,y_{k+\lambda_1-1}$ defined by
\begin{equation}
\label{eq:s_lambda(x|y)}
s_\lambda(x_1,\dots,x_k|y)
=s_{I(\lambda)}(x_1,\dots,x_k|y)
=\det(Z_{ij})_{i\in I(\lambda), j\le k},
\end{equation}
where $I(\lambda)$ is given by~\eqref{eq:I(lambda)};
cf.~\eqref{eq:s_lambda}.
To recover the ordinary Schur function, one needs to specialize the
$y$ variables to~$0$.

\begin{example}
\label{example:21-double}
Consider $\lambda=(2,1)$ with $k=2$.
Then $I(\lambda)=\{2,4\}$, and \eqref{eq:s_lambda(x|y)} becomes
\begin{align}
\notag
s_{(2,1)}(x_1,x_2|y)&=\dfrac{(x_1+y_1)(x_2+y_1)}{x_2-x_1}\det\left(
\begin{matrix}
1 & 1 \\[.2in]
(x_1+y_2)(x_1+y_3) &(x_2+y_2)(x_2+y_3)
\end{matrix}
\right)\\[.2in]
\label{eq:double-21}
&=(x_1+y_1)(x_2+y_1)(x_1+x_2+y_2+y_3).
\end{align}
\end{example}

In the special case when $I=\{\ell,\ell+1,\dots,\ell+k-1\}$
is an interval (cf.~\eqref{eq:interval}),
it is straightforward to verify that
\begin{equation}
\label{eq:special-minor-double}
s_I(x_1,\dots,x_k|y)=\det\bigl(Z_{ij}\bigr)_{\substack{\ell\le i\le \ell+k-1\\
1\le j\le k}}
=\prod_{1\le j\le k} \,\prod_{1\le b<\ell}(x_j+y_b),
\end{equation}
generalizing~\eqref{eq:special-minor}.

The algorithm(s) presented in
Section~\ref{sec:main-algorithm} can now be adapted almost verbatim
to the case of double Schur functions.
Indeed, the latter are nothing but the flag minors of the matrix~$Z$;
as such, they can be computed, in an efficient and subtraction-free
way, using the same cluster transformations as before, from the chamber
minors associated with the special pseudoline arrangement~$A^\circ$.
The only difference is in the formulas for those special minors:
here we use \eqref{eq:special-minor-double}
instead of~\eqref{eq:special-minor}.

\subsection*{Supersymmetric Schur functions}

Among many equivalent definitions of supersymmetric Schur functions
(or super-Schur functions for short),
we choose the one most convenient for our purposes,
due to I.~Goulden--C.~Greene~\cite{Goulden-Greene} and
I.~G.~Macdonald~\cite{Macdonald-92}.
We assume the reader's familiarity with the concepts of a
\emph{Young diagram} and a
\emph{semistandard Young tableau} (of some \emph{shape}~$\lambda$);
see, e.g., \cite{Macdonald, EC2} for precise definitions.

We start with a version with an infinite number of variables.
Let $x_1,x_2,\dots$ and $y_1,y_2,\dots$ be two sequences of
indeterminates.
The \emph{super-Schur function}
$s_\lambda(x_1,x_2,\dots;y_1,y_2,\dots)$
is a formal power series defined by
\begin{equation}
\label{eq:super-schur}
s_\lambda(x_1,x_2,\dots;y_1,y_2,\dots)
=\sum_{|T|=\lambda}\, \prod_{s\in\lambda}\, (x_{T(s)}+y_{T(s)+C(s)})
\end{equation}
where
\begin{itemize}
\item
the sum is over all semistandard tableaux $T$ of
shape~$\lambda$ with positive integer entries,
\item
the product is over all boxes~$s$ in the Young diagram of~$\lambda$,
\item
$T(s)$ denotes the entry of $T$ appearing in the box~$s$, and
\item
$C(s)=j-i$ where $i$ and $j$ are the row and column that $s$ is in,
  respectively.
\end{itemize}
We note that $T(s)+C(s)$ is always a positive integer,
so the notation $y_{T(s)+C(s)}$ makes~sense.

While this is not at all obvious from the above definition,
$s_\lambda(x_1,x_2,\dots;y_1,y_2,\dots)$ is symmetric in
the variables $x_1,x_2,\dots$;
%i.e., it is invariant under any permutation of
%subscripts of the $x$ variables;
it is also symmetric in $y_1,y_2,\dots$;
and is furthermore \emph{supersymmetric} as it satisfies the
cancellation rule
\[
s_\lambda(x_1,x_2,\dots;-x_1,y_2,y_3,\dots)
=s_\lambda(x_2,x_3,\dots;y_2,y_3,\dots).
\]
We will not rely on any of these facts.
We refer interested readers to
aforementioned sources for proofs and further details.

In order to define the super-Schur function in \emph{finitely} many
variables, one simply specializes the unneeded variables to~$0$.
That is, one sets
\begin{equation}
\label{eq:super-specialize}
s_\lambda(x_1,\dots,x_k;y_1,\dots,y_m)
=s_\lambda(x_1,x_2,\dots;y_1,y_2,\dots)
\bigr|_{x_{k+1}=x_{k+2}=\cdots=y_{m+1}=y_{m+2}=\cdots=0}.
\end{equation}
Note that the restriction of the set
of $x$ variables to $x_1,\dots,x_k$ cannot be achieved simply by
requiring the tableaux~$T$ in \eqref{eq:super-schur}
to have entries in $\{1,\dots,k\}$.
%This is because
A~tableau with an entry $T(s)>k$
may in fact contribute to the (specialized)
super-Schur polynomial: even though $x_{T(s)}$ vanishes under the
specialization,
$y_{T(s)+C(s)}$ does not have to.
%(In fact, the index $T(s)+C(s)$ can be rather small as $C(s)$ can be
%negative.)
See Example~\ref{example:21-super} below.

\begin{example}
[cf.\ Example~\ref{example:21-double}]
\label{example:21-super}
Let $\lambda=(2,1)$, $k=m=2$.
The relevant tableaux~$T$
(i.e., the ones contributing to the
specialization $x_3=x_4=\cdots=y_3=y_4=\cdots=0$)
are:
\[
\begin{array}{cc} 1 & \!1 \\ 2 \end{array}\qquad
\begin{array}{cc} 1 & \!2 \\ 2 \end{array}\qquad
\begin{array}{cc} 1 & \!1 \\ 3 \end{array}\qquad
\begin{array}{cc} 1 & \!2 \\ 3 \end{array}\qquad
\begin{array}{cc} 2 & \!2 \\ 3 \end{array}
\]
Then formulas \eqref{eq:super-schur} and \eqref{eq:super-specialize}
give
\begin{align*}
s_{(2,1)}(x_1,x_2;y_1,y_2)
&=(x_1+y_1)(x_2+y_1)(x_1+y_2)
+(x_1+y_1)(x_2+y_1)x_2 \\
&\quad +(x_1+y_1)y_2(x_1+y_2)
+(x_1+y_1)y_2x_2
+(x_2+y_2)y_2x_2\\
&=x_1x_2(x_1\!+\!x_2)\!+\!(x_1\!+\!x_2)^2(y_1\!+\!y_2)\!+\!(x_1\!+\!x_2)(y_1\!+\!y_2)^2\!+\!y_1y_2(y_1\!+\!y_2).
\end{align*}
Specializing further at $y_2=0$, we obtain
\begin{equation}
\label{eq:super-21}
s_{(2,1)}(x_1,x_2;y_1)
=(x_1+x_2)(x_1+y_1)(x_2+y_1).
\end{equation}
\end{example}

The close relationship between super-Schur functions and
double Schur functions was already exhibited in \cite{Goulden-Greene,
  Macdonald-92}.
For our purposes, we will need the following version of those
classical results.

We denote by $\ell(\lambda)$ the \emph{length} of a partition~$\lambda$,
i.e., the number of its nonzero parts~$\lambda_i$.

\begin{proposition}
\label{prop:super=double}
Assume that $m+\ell(\lambda)\le k+1$.
Then
\begin{equation}
\label{eq:super=double}
s_\lambda(x_1,\dots,x_k;y_1,\dots,y_m)
=s_\lambda(x_1,\dots,x_k|y)\bigr|_{y_{m+1}=y_{m+2}=\cdots=0}.
\end{equation}
\end{proposition}

To illustrate, let $\lambda=(2,1)$, $k=2$, $m=1$. Then the left-hand side
of \eqref{eq:super=double} is given by~\eqref{eq:super-21},
which matches~\eqref{eq:double-21} specialized at $y_2=y_3=0$.

The condition $m+\ell(\lambda)\le k+1$ in Proposition~\ref{prop:super=double}
cannot be dropped:
for example, \eqref{eq:super=double} is false for
$\lambda=(2,1)$ and $k=m=2$
(the right-hand side is not even symmetric in $y_1$ and~$y_2$).

\begin{proof}[Proof of Proposition~\ref{prop:super=double}]
First, it has been established in \cite[(6.16)]{Macdonald-92} that
\[
s_\lambda(x_1,\dots,x_k|y)
=\sum_{|T|=\lambda}\, \prod_{s\in\lambda}\, (x_{T(s)}+y_{T(s)+C(s)}),
\]
the sum over all semistandard tableaux $T$ with entries in
$\{1,\dots,k\}$.
Second, in the formula \eqref{eq:super-schur},
a tableau~$T$ with an entry $T(s)>k$ does not contribute
to the specialization~\eqref{eq:super-specialize}
since $T(s)+C(s)\ge k+1+1-\ell(\lambda)\ge m+1$
and consequently $x_{T(s)}+y_{T(s)+C(s)}=0$.
Hence both sides of \eqref{eq:super=double} are given by
the same combinatorial~formulae.
\end{proof}

\begin{theorem}
\label{eq:th:super}
The super-Schur polynomial $s_\lambda(x_1,\dots,x_k;y_1,\dots,y_m)$
can be computed by a subtraction-free arithmetic circuit of size $O((k+m)^3)$,
assuming that $k\ge \lambda_1+\ell(\lambda)-2$.
\end{theorem}

\begin{proof}
Denote $k^*=m+\ell(\lambda)-1$.
If $k\ge k^*$, then \eqref{eq:super=double}
holds, and we can compute %the super-Schur polynomial
$s_\lambda(x_1,\dots,x_k;y_1,\dots,y_m)$
using the subtraction-free algorithm for a double Schur function, in
time $O((k+\lambda_1)^3)$.

{}From now on, we assume that $k\le k^*-1$.
We can still use~\eqref{eq:super=double} with $k$ replaced by~$k^*$,
and then specialize the extra variables to~$0$:
\begin{equation}
s_\lambda(x_1,\dots,x_k;y_1,\dots,y_m)
=s_\lambda(x_1,\dots,x_{k^*}|y)
_{\bigr|\substack{y_{m+1}=y_{m+2}=\cdots=0\\[.03in]
x_{k+1}=\cdots=x_{k^*}=0}}
\,.
\end{equation}
The plan is to compute the right-hand side using the algorithm described above
for the double Schur functions, with some of the $x$ and $y$ variables specialized
to~$0$:
\begin{equation}
\label{eq:super-vanish}
y_{m+1}=y_{m+2}=\cdots=0, \quad
x_{k+1}=\cdots=x_{k^*}=0.
\end{equation}
In order for this version of the algorithm to work,
we need to make sure that the initial flag minors \eqref{eq:special-minor-double}---and
consequently all chamber minors computed by the algorithm---do not vanish
under~\eqref{eq:super-vanish}.
Note that we do not have to worry about the vanishing of denominators in~\eqref{eq:Zij}
since the algorithm does not rely on the latter formula.
(The specialization as such is always defined since $s_\lambda(x_1,\dots,x_{k^*}|y)$
is a polynomial.)

%only occurs \emph{after} each double Schur polynomial has been computed.

The algorithm that computes $s_\lambda(x_1,\dots,x_{k^*}|y)$ works with
(specialized) flag minors of a square matrix of size
\[
n^*=k^*+\lambda_1=m+\ell(\lambda)-1+\lambda_1\,.
\]
In the case of an initial flag minor, we have the formula
\begin{equation}
\label{eq:special-minor-specialized}
s_{[\ell,\ell+s-1]}(x_1,\dots,x_s|y)
=\prod_{1\le j\le s} \,\prod_{1\le b<\ell-1}(x_j+y_b)
\end{equation}
(cf.~\eqref{eq:special-minor-double}); here $\ell+s-1\le n^*$, the size of the matrix.
We see that such an initial minor vanishes (identically) under the specialization
\eqref{eq:super-vanish}
if and only if the factor $x_s+y_{\ell-1}$ vanishes, or equivalently $s\ge k+1$ and $\ell-1\ge m+1$.
This however cannot happen since it would imply that
\[
m+k+2\le
\ell+s-1\le
n^*=
m+\ell(\lambda)-1+\lambda_1
\]
which contradicts the condition $k\ge \lambda_1+\ell(\lambda)-2$ in the theorem.
\end{proof}

We expect the condition $k\ge \lambda_1+\ell(\lambda)-2$ in Theorem~\ref{eq:th:super}
to be unnecessary.
Note that one could artificially increase the number of $x$ variables to satisfy this condition,
then specialize the extra variables to~$0$.
Such a specialization however is not included among the operations
allowed in arithmetic circuits.

%[How about $Q$- and $P$-Schur functions?
%Describe a formula for the square of such a function.
%Discuss the isuue of subtraction-free computation of a square root.]

\section{Skew Schur functions}
\label{sec:skew}

In this section, we use the Jacobi-Trudi identity to reduce the
problem of subtraction-free computation of a skew Schur function to
the analogous problem for the ordinary Schur functions.
This enables us to deduce Theorem~\ref{th:skew} from
Theorem~\ref{th:schur}.

In accordance with usual conventions \cite{Macdonald, EC2},
we denote by $h_m(x_1,\dots,x_k)$
the complete homogeneous symmetric polynomial of degree~$m$.
For $m<0$, one has $h_m=0$ by definition.

Let $\lambda=(\lambda_1,\dots,\lambda_k)$ and
$\nu=(\nu_1,\dots,\nu_k)$ be partitions with at most $k$ parts.
The \emph{skew Schur function} $s_{\lambda/\nu}(x_1,\ldots, x_k)$
can be defined by the \emph{Jacobi-Trudi formula}
\begin{equation}
\label{eq:jacobi-trudi}
s_{\lambda/\nu}(x_1,\ldots, x_k)=\det(h_{\lambda_i-\nu_j-i+j}(x_1,\ldots, x_k)).
\end{equation}
The polynomial $s_{\lambda/\nu}(x_1,\ldots, x_k)$ is nonzero
if and only if $\nu_i\le \lambda_i$ for all~$i$;
the latter condition is abbreviated by $\nu\subset\lambda$.

Formula~\eqref{eq:jacobi-trudi} can be rephrased as saying that $s_{\lambda/\nu}$
is the $k\times k$ minor of the infinite Toeplitz matrix $(h_{i-j})$
that has row set $I(\lambda)$ (see~\eqref{eq:I(lambda)}) and
column set~$I(\nu)$.

Let $n>k$.
We fix the partition~$\nu$, and let $\lambda$ vary over all partitions
satisfying $I(\lambda)\subset\{1,\dots,n\}$,
or equivalently $k+\lambda_1\le n$.
Let us denote by $H_\nu$ the $n\times k$ matrix
\[
H_\nu=(h_{i-j}(x_1,\ldots, x_k))_{\substack{1\le i\le n\\ j\in I(\nu)}}\,.
\]
The maximal (i.e., $k\times k$) minors of $H_\nu$ are the (possibly vanishing)
skew Schur polynomials $s_{\lambda/\nu}(x_1,\ldots, x_k)$.
More generally, a $p\times p$ flag minor of $H_\nu$
is a skew Schur polynomial of the form $s_{\lambda/\nu(p)}$
where $\lambda$ is a partition with at most $p$ parts
satisfying $p+\lambda_1\le n$,
and $\nu(p)=(\nu_{k-p+1},\dots,\nu_k)$
denotes the partition formed by $p$ smallest (possibly zero) parts of~$\nu$.
Such a flag minor does not vanish if and only if $\nu(p)\subset\lambda$.

Our algorithm computes a skew Schur polynomial $s_{\lambda/\nu}(x_1,\ldots, x_k)$
(equivalently, a maximal minor of~$H(\nu)$)
using the same approach as before:
we first compute the initial flag minors corresponding to intervals~\eqref{eq:interval},
then proceed via recursive cluster transformations.

The problem of calculating the interval flag minors of $H_\nu$
(in an efficient and subtraction-free way)
turns out to be equivalent to the (already solved)
problem of computing ordinary Schur polynomials.
This is because $I(\lambda)$ is an interval if and only if $\lambda$ has \emph{rectangular shape},
i.e., all its nonzero parts are equal to each other.
For such a partition, the nonzero skew Schur polynomial $s_{\lambda/\nu}$
is well known to coincide with an ordinary Schur polynomial $s_\theta$
where $\theta$ is the partition formed by the differences $\lambda_i-\nu_i$.

\iffalse
Consider an interval $a, a+1,\ldots, a+j$. The
corresponding interval minor is the determinant of a submatrix
located in the intersection of the first $j+1$ rows  the columns
from the interval $a, a+1,\ldots, a+j$, and  it is equal to the
Schur function $s_{a^j/ (\nu_{n-j},\ldots,
\nu_n)}=s_{a-\nu_n,\ldots, a-\nu_{n-j}}$, where $a^j=(a,\ldots ,a)$
is the rectangle $a\times (j+1)$ and $(\nu_{n-j},\ldots, \nu_n)$ is
the last $j+1$ parts of the partition $\nu$.

Let us remark  that if the diagram $(\nu_{n-j},\ldots, \nu_n)$ can
not be embedded into rectangle $a\times (j+1)$, then
$s_{a^j/ (\nu_{n-j},\ldots, \nu_n)}=0$.

By Theorem~\ref{th:schur}, we get a subtraction free
computation of any Schur function $ s_{a-\nu_n,\ldots, a-\nu_{n-j}}$
with bit complexity $O(n^3\log^2 n)$ .
\fi

We then proceed, as before,
with a recursive computation utilizing cluster transformations.
However, substantial adjustments have to be made due to the fact that many
flag minors of $H_\nu$ vanish.
(Also, $H_\nu$ is not a square matrix, but this issue is less important.)
Our recipe is as follows.
Suppose that we need to perform a step of our algorithm that involves,
in the notation of Figure~\ref{fig:move},
expressing $f$ in terms of $a,b,c,d,e$.
(It is easy to see that we never have to move in the opposite direction,
i.e., from $a,b,c,d,f$ to~$e$, while moving away from the special arrangement~$A^\circ$
using the algorithms described above.)
If $e\neq 0$ (and we shall know beforehand whether this is the case or not),
then set $f=(ac+bd)/e$ as before.
If, on the other hand, $e=0$, then set~$f=0$.

In order to justify this algorithm, we need to show that the skew Schur polynomials
at hand have the property $e=0\Rightarrow f=0$, in the above notation.
(Also, it is not hard to check in the process of computing a flag minor of size~$k$,
we never need to compute a flag minor of larger size which would not fit into~$H_\nu$.)
This property is a rather straightforward consequence of the criterion
for vanishing/nonvanishing of skew Schur functions.
Let $p<q<r$ denote the labels of the lines shown in Figure~\ref{fig:move},
and let $J$ denote the set of lines passing below the shown fragment.
Then $e=s_{J\cup \{q\}}$ and $f=s_{J\cup \{p.r\}}$.
Since $p<q$, the vanishing of~$e$ implies the vanishing of
$s_{J\cup \{p\}}$, which in turn implies the vanishing of $f=s_{J\cup \{p.r\}}$.
We omit the details.

The complexity of the algorithm is dominated by the initialization stage,
which involves computing $O(n^2)$ ordinary Schur polynomials;
each of them takes $O(n^3)$ operations to compute.
The bit complexity is accordingly $O(n^5\log^2 n)$.

\section{Generating functions for spanning trees
%subtraction-free computation
}
\label{sec:spanning-trees}

In this section, we present a polynomial subtraction-free algorithm for computing
the generating function for \emph{spanning trees} in a graph with weighted edges
(a~\emph{network}).
%We begin by setting up the notation,
%and by explaining the relevance of this problem to
%the questions posed in Problem~\ref{problem:M-m} and Example~\ref{example:sf-complexity}.
While this algorithm is going to be improved upon in Section~\ref{sec:directed-spanning-trees},
we decided to include it because of its simplicity, and in order to highlight the connection
to the theory of electric networks (equivalently, discrete potential theory).
An impatient reader can go straight to Section~\ref{sec:directed-spanning-trees}.

Let $G$ be an undirected connected graph with vertex set $V$ and edge set~$E$.
We associate a variable $x_e$ to each edge $e\in E$,
and consider the generating function $f_G$ (a polynomial in the variables~$x_e$)
defined by
\[
f_G = \sum_T x^T
\]
where the summation is over all spanning trees $T$ for~$G$,
and $x^T$ denotes the product of the variables $x_e$ over all edges $e$ in~$T$.
An example is given in Figure~\ref{fig:f_G}.

\begin{figure}[ht]
\begin{center}
\setlength{\unitlength}{2.2pt}
\begin{picture}(140,22)(0,1)
\thinlines
\put(0,0){\line(1,0){20}}
\put(0,0){\line(0,1){20}}
\put(20,0){\line(0,1){20}}
\put(20,0){\line(-1,1){20}}
\put(0,20){\line(1,0){20}}

\put(0,0){\circle*{2}}
\put(0,20){\circle*{2}}
\put(20,0){\circle*{2}}
\put(20,20){\circle*{2}}

\put(-4,0){\makebox(0,0){$\mathbf{1}$}}
\put(-4,20){\makebox(0,0){$\mathbf{2}$}}
\put(24,0){\makebox(0,0){$\mathbf{4}$}}
\put(24,20){\makebox(0,0){$\mathbf{3}$}}

\put(-4,10){\makebox(0,0){$x_{12}$}}
\put(25,10){\makebox(0,0){$x_{34}$}}
\put(10,-3){\makebox(0,0){$x_{14}$}}
\put(12,13){\makebox(0,0){$x_{24}$}}
\put(10,23){\makebox(0,0){$x_{23}$}}

\put(40,8){$\begin{array}{rcl}
f_G\!\!&=&\!\!\!x_{12}x_{14}x_{23}+x_{12}x_{14}x_{34}+x_{12}x_{23}x_{24}+x_{12}x_{23}x_{34}\\[.05in]
&\,+&\!\!\!x_{12}x_{24}x_{34}+x_{14}x_{23}x_{24}+x_{14}x_{23}x_{34}+x_{14}x_{24}x_{34}
\end{array}$}

\end{picture}
\end{center}
\caption{A weighted graph $G$ and the spanning tree generating function~$f_G$}
\label{fig:f_G}
\end{figure}
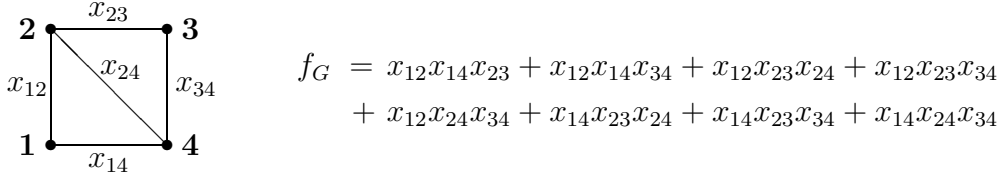

\begin{remark}
\label{rem:simple-graph}
Without loss of generality, we may restrict ourselves to the case when the graph~$G$ is \emph{simple},
that is, $G$~has neither loops (i.e., edges with coinciding endpoints) nor multiple edges.
Loops cannot contribute to a spanning tree, so we can throw them away without altering~$f_G$.
Furthermore, if say vertices $v$ and $w$ are connected by several edges
$e_1,\dots,e_\ell$, then we can replace them by a single edge of weight
$x_{e_1}+\cdots +x_{e_\ell}$ without changing the generating function~$f_G$.
\end{remark}

\pagebreak[3]

Recall that the number of spanning trees in a complete graph on $n$ vertices is equal to~$n^{n-2}$,
so the monomial expansion of $f_G$ may have a superexponential number of terms.
On the other hand, there is a well-known determinantal formula for~$f_G$,
due to G.~Kirchhoff~\cite{Kirchhoff} (see, e.g., \cite[Theorem~II.12]{Bollobas}),
known as the (weighted) Matrix Tree Theorem.
This formula provides a way to compute~$f_G$ in polynomial time---but the calculation
involves subtraction.
Is there a way to efficiently compute $f_G$ using only addition, multiplication, and division?
Just like in the case of Schur functions,
the answer turns out to be \emph{yes}.

\begin{theorem}
\label{th:spanning-trees-sf}
In a weighted simple graph $G$ on $n$ vertices,
the spanning tree generating function $f_G$ can be computed by a subtraction-free
arithmetic circuit of size~$O(n^4)$.
\end{theorem}

This result is improved to $O(n^3)$ in Section~\ref{sec:directed-spanning-trees}.

The rest of this section is devoted to the proof of Theorem~\ref{th:spanning-trees-sf},
i.e., the description of an algorithm that computes $f_G$ using $O(n^4)$
additions, multiplications, and divisions.
The algorithm utilizes well known techniques
from the theory of electric networks (more precisely, circuits made of ideal resistors).
In order to apply these techniques to the problem at hand, we interpret each edge weight~$x_e$
as the electrical \emph{conductance} of~$e$,
i.e., the inverse of the resistance of~$e$.
We note that the rule, discussed in Remark~\ref{rem:simple-graph},
for combining parallel edges into a single edge is compatible with this interpretation.

\pagebreak[3]

\begin{definition}[Gluing two vertices]
\label{def:gluing-vv'}
Let $v$ and $v'$ be distinct vertices in a weighted simple graph~$G$ as above.
We denote by $G(v,v')$ the weighted simple graph obtained from~$G$ by
\begin{itemize}
\item[(i)]
gluing together the vertices $v$ and~$v'$ into a single vertex which we call~$\boxed{vv'}\,$, then
\item[(ii)]
removing the loop at~$\boxed{vv'}$ (if any), and then
\item[(iii)]
for each vertex $u$ connected in $G$ to both $v$ and~$v'$, say by edges $e$ and~$e'$,
replacing $e$ and $e'$ by a single edge of conductance $x_e+x_{e'}$ between $u$ and~$\boxed{vv'}\,$.
\end{itemize}
In view of Remark~\ref{rem:simple-graph}, steps (ii) and~(iii) do not change
the spanning tree generating function of the graph at hand.
An example is shown in Figure~\ref{fig:f_G12}.
\end{definition}

%\vspace{-.05in}

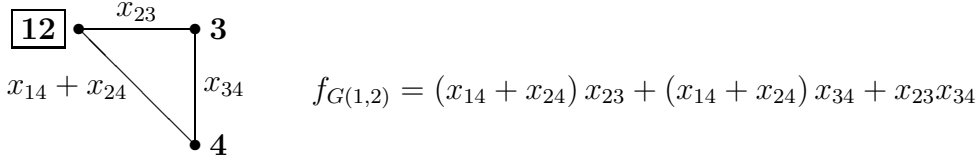
\begin{figure}[ht]
\begin{center}
\setlength{\unitlength}{2.2pt}
\begin{picture}(140,23)(0,0)
\thinlines
%\put(0,0){\line(1,0){20}}
%\put(0,0){\line(0,1){20}}
\put(20,0){\line(0,1){20}}
\put(20,0){\line(-1,1){20}}
\put(0,20){\line(1,0){20}}

%\put(0,0){\circle*{2}}
\put(0,20){\circle*{2}}
\put(20,0){\circle*{2}}
\put(20,20){\circle*{2}}

%\put(-4,0){\makebox(0,0){$\mathbf{1}$}}
\put(-7,20){\makebox(0,0){$\boxed{\mathbf{12}}$}}
\put(24,0){\makebox(0,0){$\mathbf{4}$}}
\put(24,20){\makebox(0,0){$\mathbf{3}$}}

%\put(-4,10){\makebox(0,0){$x_{12}$}}
\put(25,10){\makebox(0,0){$x_{34}$}}
%\put(10,-3){\makebox(0,0){$x_{14}$}}
\put(-2,10){\makebox(0,0){$x_{14}+x_{24}$}}
\put(10,23){\makebox(0,0){$x_{23}$}}

\put(40,8){$f_{G(1,2)}
=(x_{14}+x_{24})\,x_{23}+(x_{14}+x_{24})\,x_{34}+x_{23} x_{34}
%=x_{14}x_{23}+x_{14}x_{34}+x_{23}x_{24}+x_{24}x_{34}+x_{23} x_{34}
$}

\end{picture}
\end{center}
\caption{The graph $G(1,2)$ for the graph $G$ in Figure~\ref{fig:f_G}.}
\label{fig:f_G12}
\end{figure}

%\vspace{-.05in}

\begin{lemma}[{Kirchhoff's effective conductance formula~\cite{Kirchhoff};
see, e.g., \cite[Section~2]{Wagner}}]
\label{lem:eff-conductance}
Let~$G$ be a weighted connected simple graph whose edge weights are interpreted as electrical conductances.
The~effective conductance between vertices $v$ and $v'$ of $G$ is given by
\[
\operatorname{effcond}_G(v,v')=\frac{f_G}{f_{G(v,v')}}.
\]
\end{lemma}

To illustrate, the effective conductance between vertices $1$ and~$2$ in the graph shown in
Figure~\ref{fig:f_G} is equal to
$\frac{f_G}{f_{G(1,2)}}$, where $f_G$ and $f_{G(1,2)}$ are given in Figures~\ref{fig:f_G}
and~\ref{fig:f_G12}, respectively.
This matches the formula
\[
\operatorname{effcond}_G(1,2)=x_{12}+\dfrac{1}{\textstyle\frac{\textstyle 1}{\textstyle x_{14}}
+\textstyle\frac{\textstyle 1}{\textstyle x_{24}
+\textstyle\frac{\textstyle 1}{
  \textstyle\frac{\textstyle 1}{\textstyle x_{23}}+\textstyle\frac{\textstyle 1}{\textstyle x_{34}}
}}
}
\]
that can be obtained using the \emph{series-parallel} property of this particular graph~$G$.

\begin{definition}
Let $G$ be a weighted connected simple graph on the vertex set $\{1,\dots,n\}$.
Define the graphs $G_1,\dots,G_n$ recursively by $G_1=G$ and
\[
G_{i+1}=G_i(\boxed{1\cdots i}\,,i+1)
\]
where $\boxed{1\cdots i}$ denotes the vertex obtained by gluing together the original
vertices $1,\dots,i$.
In other words, $G_i$ is obtained from $G$ by collapsing the vertices $1,\dots,i$ into
a single vertex, removing the loops, and combining multiple edges into single
ones while adding their respective weights, cf.\ Remark~\ref{rem:simple-graph}.
\end{definition}

For example, if $G$ is the graph in Figure~\ref{fig:f_G},
then $G_1\!=\!G$; $G_2$ is the graph shown in Figure~\ref{fig:f_G12};
$G_3$ is a two-vertex graph with a single edge of weight $x_{14}+x_{24}+x_{34}$;
and $G_4$ (and more generally~$G_n$) is a single-vertex graph with no edges
(so $f_{G_n}=1$).

%We note that $G_n$ is a single-vertex graph with no edges; thus $f_{G_n}=1$.

The following formula is immediate from Lemma~\ref{lem:eff-conductance}, via telescoping.

\begin{corollary}
\label{cor:telescoping}
Let $G$ be a weighted connected simple graph on the vertex set $\{1,\dots,n\}$.
Then
\[
f_G=\prod_{i=1}^{n-1} \operatorname{effcond}_{G_i}(i,i+1).
\]
\end{corollary}

Corollary~\ref{cor:telescoping} reduces the computation of the generating function~$f_G$
to the problem of computing effective conductances.
The latter can be done, both efficiently and in a subtraction-free way,
using the machinery of \emph{star-mesh transformations} developed by electrical engineers,
see, e.g., \cite[Corollary~4.21]{WKChen}.
The technique goes back at least 100 years, cf.\ the historical discussion in~\cite{Riordan}.

\begin{lemma}[{Star-mesh transformation}]
\label{lem:star-mesh}
Let $v$ be a vertex in a weighted simple graph~$G$ (viewed as an electric network with
the corresponding conductances).
Let $e_1,\dots,e_k$ be the full list of edges incident to~$v$;
assume that they connect $v$ to distinct vertices $v_1,\dots,v_k$, respectively.
Transform $G$ into a new weighted graph $G'$ defined as follows:
\begin{itemize}
\item
remove vertex $v$ and the edges $e_1,\dots,e_k$ incident to it;
\item
for all $1\le i<j\le k$, introduce a new edge $e_{ij}$ connecting $v_i$ and~$v_j$,
and assign
\begin{equation}
\label{eq:star-mesh}
x_{e_{ij}} \stackrel{\rm def}{=} x_{e_i} x_{e_j} \sum_{\ell=1}^k \frac{1}{x_{e_\ell}}
\end{equation}
as its weight (=conductance);
\item
in the resulting graph, combine parallel edges into single ones,
as in Remark~\ref{rem:simple-graph}.
\end{itemize}
Then the weighted graphs $G$ and $G'$ have the same effective conductances.
More precisely, for any pair of vertices $a,b$ different from~$v$, we have
\[
\operatorname{effcond}_G(a,b)=\operatorname{effcond}_{G'}(a,b).
\]
\end{lemma}

Lemma~\ref{lem:star-mesh} provides an efficiently way
to compute an effective conductance between two given vertices $a$ and~$b$ in a graph~$G$,
by iterating the star-mesh transformations~\eqref{eq:star-mesh}
for all vertices $v\notin\{a,b\}$, one by one.
Since these transformations are subtraction-free,
and require $O(n^2)$ arithmetic operations each,
we arrive at the following result.

\begin{corollary}
\label{cor:eff-cond-via-star-mesh}
An effective conductance between two given vertices in an $n$-vertex weighted
simple graph~$G$ can be computed by a subtraction-free arithmetic circuit of size~$O(n^3)$.
\end{corollary}

Combining Corollaries~\ref{cor:telescoping} and~\ref{cor:eff-cond-via-star-mesh},
we obtain a proof of Theorem~\ref{th:spanning-trees-sf}.
The algorithm computes the effective conductances $\operatorname{effcond}_{G_i}(i,i+1)$
for $i=1,\dots,n-1$ using star-mesh transformations, then multiplies them to get
the generating function~$f_G$.

\section{Directed spanning trees}
\label{sec:directed-spanning-trees}

In this section, we treat the directed version of the problem considered in
Section~\ref{sec:spanning-trees},
designing a polynomial subtraction-free algorithm that computes
the generating function for directed spanning trees in a directed graph with weighted edges.

Similarly to the unoriented case,
our approach makes use of the appropriate version of star-mesh transformations.
As before, they are local modifications of the network
which transform the weights by means of certain subtraction-free formulas.
There is also a difference: unlike in Section~\ref{sec:spanning-trees},
we apply these transformations directly to the computation of the generating functions
of interest---rather than to ``effective conductances'' from which those generating functions
can be recovered via telescoping.
Adapting the latter technique to the directed case
would require a thorough review of W.~Tutte's theory of ``unsymmetrical electricity''
\cite[Sections VI.4--VI.5]{Tutte-GT} \cite[Section~4]{Tutte-1998}.
This elementary but somewhat obscure theory goes back to the 1940s,
see references in \emph{loc.\ cit.},
and is closely related to Tutte's directed version of the Matrix-Tree Theorem
\cite[Theorem~6.27]{Tutte-GT}.

\begin{remark}
\label{rem:undir->dir}
The approach used in this section can be applied in the undirected case as well,
bypassing the use of electric networks
(cf.\ Section~\ref{sec:spanning-trees}).
Also, one can reduce the undirected case to the directed one
by replacing each edge~$a\stackrel{e}{\mbox{---}}b$ in an ordinary weighted graph
by two oriented edges $a\rightarrow b$ and $b\rightarrow a$ each having the weight~$x_e$
of the original~edge.
\end{remark}

In this section, $G$ is a \emph{directed graph} with vertex set $V$ and edge set~$E$,
and with a fixed vertex $r\in V$ called the \emph{root}.
A \emph{directed spanning tree} $T$ in~$G$
(sometimes called an \emph{in-tree}, an \emph{arborescence},
or a \emph{branching}) is a subgraph of $G$ that spans all vertices in~$V$
and includes a subset of edges such that for any $v\in V$, there is a unique path in~$T$ that begins
at~$v$ and ends at~$r$.
Equivalently, $T$ is a spanning tree of~$G$ %(if one ignores the orientations of edges)
in which all edges are oriented towards~$r$.

We assume that $G$ has at least one such tree,
or equivalently that there is a path from any vertex $v\in V$ to the root~$r$.

We associate a variable $x_e$ to each (directed) edge $e\in E$,
and define the generating function $\varphi_G$ by
\[
\varphi_G = \sum_T x^T
\]
where the summation is over all directed spanning trees $T$ for~$G$ (rooted at~$r$).
As before, $x^T$ denotes the product of the variables $x_e$ over all edges $e$ in~$T$.
Figure~\ref{fig:varphi_G} shows the generating function $\varphi_G$ for the
\emph{complete directed graph} on three vertices.

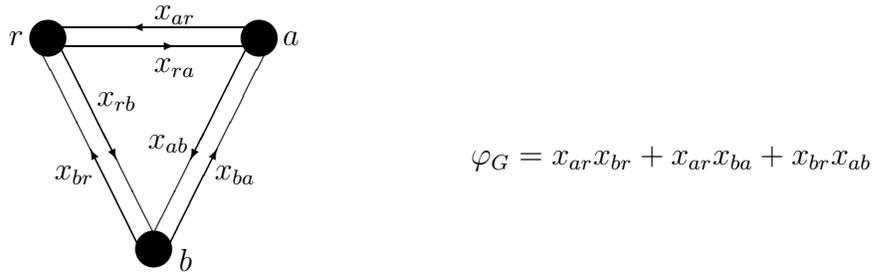
\begin{figure}[ht]
\begin{center}
\setlength{\unitlength}{4pt}
\begin{picture}(70,22)(0,1)
\thinlines
\put(0,21){\line(1,0){20}}
\put(0,19.2){\line(1,0){20}}
\put(9,-.5){\line(-1,2){10}}
\put(10.5,.5){\line(-1,2){10}}
\put(11,-.5){\line(1,2){10}}
\put(9.5,.5){\line(1,2){10}}

\put(0,19.2){\vector(1,0){12}}
\put(20,21){\vector(-1,0){12}}
\put(9,-.5){\vector(-1,2){5}}
\put(.5,20.5){\vector(1,-2){6}}
\put(11,-.5){\vector(1,2){5}}
\put(19.5,20.5){\vector(-1,-2){6}}

\put(0,20){\circle*{3.5}}
\put(20,20){\circle*{3.5}}
\put(10,0){\circle*{3.5}}

\put(-3,20){\makebox(0,0){$r$}}
\put(23,20){\makebox(0,0){$a$}}
\put(13,-1){\makebox(0,0){$b$}}

\put(12,22.2){\makebox(0,0){$x_{ar}$}}
\put(12,17.2){\makebox(0,0){$x_{ra}$}}
\put(6.5,14){\makebox(0,0){$x_{rb}$}}
\put(2.5,7){\makebox(0,0){$x_{br}$}}
\put(11.4,10){\makebox(0,0){$x_{ab}$}}
\put(17.7,7){\makebox(0,0){$x_{ba}$}}

\put(40,8){$\varphi_G=x_{ar}x_{br}+x_{ar}x_{ba}+x_{br}x_{ab}$}

\end{picture}
\end{center}
\caption{The generating function~$\varphi_G$ for the directed spanning trees in~$G$.}
\label{fig:varphi_G}
\end{figure}

Without loss of generality, we may assume that $G$ is a \emph{simple} directed graph,
i.e., it has no loops and no multiple edges, for the same reasons as in Remark~\ref{rem:simple-graph}.
We certainly do allow pairs of edges connecting the same pair of vertices but oriented in
opposite ways.

\begin{theorem}
\label{th:directed-spanning-trees-sf}
In a weighted simple directed graph $G$ on $n$ vertices, %with no multiple edges,
the generating function for directed spanning trees rooted at a given vertex~$r$
can be computed by a subtraction-free
arithmetic circuit of size~$O(n^3)$.
\end{theorem}

In view of Remark~\ref{rem:undir->dir}, the analogue of Theorem~\ref{th:directed-spanning-trees-sf}
for undirected graphs follows, improving upon Theorem~\ref{th:spanning-trees-sf}
and implying Theorem~\ref{th:combined-spanning-trees}.

The algorithm that establishes Theorem~\ref{th:directed-spanning-trees-sf}
relies on the following lemma.

\begin{lemma}[{Star-mesh transformation in a directed network}]
\label{lem:star-mesh-directed}
Let $v\neq r$ be a vertex in a weighted directed graph~$G$ as above.
Let $v_1,\dots,v_k$ be the full list of vertices directly connected to~$v$
by an edge (either incoming, or outgoing, or both).
For $i=1,\dots,k$, let~$x_i$ (resp.,~$y_i$)
denote the weight of the edge $v_i\rightarrow v$ (resp., $v\rightarrow v_i$);
in the absence of such edge, set $x_i\!=\!0$ (resp., $y_i\!=\!0$).
Transform $G$ into a new weighted directed graph $G''$ as follows:
\begin{itemize}
\item
remove vertex $v$ and all the edges incident to it;
\item
for each pair $i,j\in\{1,\dots,k\}$, $i\neq j$, $x_i y_j\neq 0$,
introduce a new edge $e_{ij}$ directed from $v_i$ to~$v_j$,
and set its weight to be
\begin{equation}
\label{eq:directed-star-mesh}
x_{e_{ij}} \stackrel{\rm def}{=} x_i \,y_j \,(y_1+\cdots+y_k)^{-1};
\end{equation}
\item
in the resulting graph~$G'$, combine multiple edges (if any), adding their respective weights,
to obtain~$G''$. (Thus $\varphi_{G''}=\varphi_{G'}\,$.)
\end{itemize}
Then
\begin{equation}
\label{eq:directed-star-mesh-gf}
\varphi_G= (y_1+\cdots+y_k)\,\varphi_{G''}\,.
\end{equation}
\end{lemma}

We note that $y_1+\cdots+y_k\neq 0$ since otherwise there is no path from $v$ to~$r$.
(If that happens, we have $\varphi_G=0$.)

It is easy to see that Lemma~\ref{lem:star-mesh-directed} implies
Theorem~\ref{th:directed-spanning-trees-sf}.
The algorithm computes the generating function $\varphi_G$ by iterating the star-mesh
transformations described in the lemma.

\begin{example}
Consider the weighted graph in Figure~\ref{fig:varphi_G}.
Choose $v=b$. The recipe in Lemma~\ref{lem:star-mesh-directed}
asks us to remove the vertex~$b$ and the four edges incident to it,
introducing instead two edges connecting $r$ and~$a$.
According to the formula~\eqref{eq:directed-star-mesh}, the new edge in $G'$
pointing from $a$ to~$r$
has weight $x_{ab}\,x_{br}(x_{ba}+x_{br})^{-1}$.
Adding this to the weight $x_{ar}$ of the old edge $a\rightarrow r$,
we obtain the combined weight of the edge going from $a$ to~$r$ in the
two-vertex graph~$G''$.
Thus
\[
\varphi_{G''}=x_{ar}+\frac{x_{ab}\,x_{br}}{x_{ba}+x_{br}}
=\frac{x_{ar}\,x_{ba}+x_{ar}\,x_{br}+x_{ab}\,x_{br}}{x_{ba}+x_{br}}.
\]
Then \eqref{eq:directed-star-mesh-gf} gives
\[
\varphi_G= (x_{ba}+x_{br})\,\varphi_{G''}=x_{ar}\,x_{ba}+x_{ar}\,x_{br}+x_{ab}\,x_{br}\,,
\]
matching the result of a direct calculation in Figure~\ref{fig:varphi_G}.
\end{example}

It remains to prove Lemma~\ref{lem:star-mesh-directed}.
The proof uses a classical result
(see, e.g., \cite[Theorem~5.3.4]{EC2}, with $k\!=\!1$)
sometimes called ``the Cayley-Pr\"ufer theorem;'' it is indeed immediate from Pr\"ufer's
celebrated proof of Cayley's formula for the number of spanning trees.
We state this result in a version best suited for our purposes.

\begin{lemma}
\label{lem:cayley-prufer}
Let $H$ be a complete directed graph on the vertex set~$W$, with root~$r\in W$.
For $v\!\in\! W$, let $z_v$ be a formal variable.
Assign to every edge $a\rightarrow b$ in $H$ the weight~$z_b\,(\sum_v z_v)^{-1}$.
Then substituting these weights into $\varphi_H$ gives $z_r\,(\sum_v z_v)^{-1}$.
\end{lemma}

\begin{example}
The case when $H$ has three vertices is shown in Figure~\ref{fig:varphi_G}.
Substituting $x_{ij}=z_j\,(z_a+z_b+z_r)^{-1}$, we get
\[
\varphi_H=x_{ar}x_{br}+x_{ar}x_{ba}+x_{br}x_{ab}
=(z_r^2+z_rz_a+z_rz_b)(z_a+z_b+z_r)^{-2}
=z_r(z_a+z_b+z_r)^{-1}.
\]
\end{example}

\begin{proof}[Proof of Lemma~\ref{lem:star-mesh-directed}]
The proof uses standard techniques of elementary enumerative combinatorics.
As equation~\eqref{eq:directed-star-mesh-gf} is equivalent to
\begin{equation}
\label{eq:directed-star-mesh-gf-G'}
\varphi_G= (y_1+\cdots+y_k)\,\varphi_{G'}\,,
\end{equation}
we will be proving the latter identity.

The edge set $E$ of $G$ naturally splits into two disjoint subsets.
The $2k$ edges $v_i\rightarrow v$ and $v\rightarrow v_i$
form $\operatorname{Star}_v$ (the \emph{star}  of~$v$).
The remaining edges form the set $\operatorname{Out}_v=E\setminus\operatorname{Star}_v$.
Similarly, the edge set $E'$ of $G'$ is a disjoint union of
$\operatorname{Mesh}_v\!=\!\{e_{ij}\}$ (the \emph{mesh} of~$v$)
and~$\operatorname{Out}_v\,$.

We shall write $\varphi_G$ (resp.,~$\varphi_{G'}$) as a sum of terms of the form
$AB$ where $A$ is a polynomial expression in the weights of the edges in $\operatorname{Star}_v$
(resp., $\operatorname{Mesh}_v$) while $B$ only involves the weights of
edges in~$\operatorname{Out}_v$.
Each factor $B$ will be a generating function for a certain class of directed forests
in~$\operatorname{Out}_v$.
(Think of those forests as leftover chunks of a directed tree~after~its edges in
$\operatorname{Star}_v$
(resp., $\operatorname{Mesh}_v$) have been removed.)
More specifically, the factors $B$ in our formulas will be of the following kind.
Let $\Pcal\!=\!\{P_a\}$ be an (unordered) partition of the set
\[
K=\{v_1,\dots,v_k\}\cup\{r\}
\]
into nonempty subsets~$P_a$ (called \emph{blocks}) where in each block~$P_a$,
one vertex $a$ has been designated as the \emph{root} of the block.
If $r\in P_a$ (i.e., if the block contains the root of~$G$), then we require that $a=r$;
moreover $P_r$ must contain at least one of the elements $v_1,\dots,v_k$.
We denote by $B(\Pcal)$ the generating function for the directed forests~$F$
which span the vertex set~$V\setminus\{v\}$ and
have the property that the vertices in~$K$
are distributed among the connected components of~$F$ as prescribed by~$\Pcal$.
More precisely, each connected component $C$ of~$F$ is a directed tree whose vertex set
includes all vertices from some block~$P_a$ of~$\Pcal$ (and no
vertices from other blocks),
with $a$ serving as the root of~$C$.
(In particular, $C$ contains at least one of the vertices $v_1,\dots,v_k$.)
The weight of $F$ is the product of the weights of its edges.

To complete the proof, we are going to write formulas of the form
\begin{align}
\label{eq:AB}
\varphi_G&=\textstyle\sum_\Pcal A(\Pcal) B(\Pcal)\\
\label{eq:A'B}
\varphi_{G'}&=\textstyle\sum_\Pcal A'(\Pcal) B(\Pcal)
\end{align}
(sums over rooted set partitions $\Pcal$ as above)
and demonstrate that for any~$\Pcal$, we have
\begin{equation}
\label{eq:A=YA'}
A(\Pcal)=(y_1+\cdots+y_k)\,A'(\Pcal).
\end{equation}

Let $\Pcal=\{P_a\}$ be a partition of $K$ as above.
For each block $P_a$,
%denote $X_P=x_{r_P}$, the weight of the edge $r_P\rightarrow v$ in $\operatorname{Star}_v$. Also
denote $Y_a=\sum_{v_i\in P_a} y_i$, the sum of the weights
of the edges $v\rightarrow v_i$ entering the block~$P_a$.
The edges of each directed tree in~$G$ contributing to $\varphi_G$
split into those contained in $\operatorname{Star}_v$
and those belonging to $\operatorname{Out}_v\,$.
The latter edges form a directed spanning forest in~$V\setminus\{v\}$
whose connected components, with their roots identified,
correspond to a partition~$\Pcal$ as above.
Direct inspection shows that combining the terms in $\varphi_G$
corresponding to each~$\Pcal$ yields the formula~\eqref{eq:AB} with
\[
A(\Pcal)=Y_r\prod_{a\neq r} x_a\,.
\]
An analogous---if less straightforward---calculation for the graph~$G'$,
with $\operatorname{Star}_v$ replaced \linebreak[3]
by~$\operatorname{Mesh}_v$,
results in the formula~\eqref{eq:A'B} with
\[
A'(\Pcal)=\sum_T \prod_{P_a\rightarrow P_b} x_a\,Y_b\,(y_1+\cdots+y_k)^{-1}
\,,
\]
where the sum is over all directed trees~$T$ on the vertex
set~$\{P_a\}$,
with root~$P_r$ 
(i.e., the vertices of $T$ are the blocks of~$\Pcal$),
and the product is over all directed edges $P_a\rightarrow P_b$ in~$T$.
We note that $\sum_{P_a} Y_a=y_1+\cdots+y_k$.
Thus Lemma~\ref{lem:cayley-prufer} applies, and we get
\[
A'(\Pcal)=Y_r\, (y_1+\cdots+y_k)^{-1}\prod_{a\neq r} x_a\,,
\]
implying~\eqref{eq:A=YA'}.
\end{proof}

%\pagebreak

\section{Subtraction-free complexity vs.\ ordinary complexity}
\label{sec:quadratic}

In this section, we exhibit a sequence of rational functions $(f_n)$
%(indeed, univariate polynomials)
whose ordinary arithmetic circuit complexity is linear in~$n$
(or even $O(1)$ if one allows arbitrary constants as inputs)
while their subtraction-free complexity grows exponentially in~$n$.

\begin{lemma}
\label{lem:denom-to-complexity}
Let $F$ be a rational function (in one or several variables) 
representable as a ratio of polynomials with nonnegative coefficients.
Assume that in any such representation $F=P/Q$,
the (total) degree of~$P$ is greater than~$2^m$. 
Then the subtraction-free complexity of~$F$ is greater than~$m$.
\end{lemma}

\begin{proof}
Let $\mathcal{D}_k$ denote the class of rational functions $f$
which can be written in the form $f=p/q$ where both $p$
and $q$ have nonnegative coefficients and have degrees at
most~$k$. It is easy to see that if $f_1,f_2\in \mathcal{D}_k$,
then each of the functions $f_1+f_2$, $f_1 f_2$, and $f_1/f_2$
lie~in~$\mathcal{D}_{2k}$. It follows that if $F$ has
subtraction-free complexity~$l$, then $F\in\mathcal{D}_{2^l}(x)$.
On the other hand, the conditions in the lemma imply that
$F\notin\mathcal{D}_{2^m}$. Hence $l>m$.
\end{proof}

%The following result is both simple and not new, cf.~\cite[p.~222]{Powers-Reznick}.
%We include a proof for the sake of completeness.

\begin{lemma}
\label{lem:F_alpha-is-SF}
%If $\cos(\alpha)<1$, then $F_\alpha(x)$
For a positive integer~$N$, the quadratic univariate polynomial 
\[
F_N(x)=(x-1)^2+\frac{1}{N^2} %\in \RR[x]
\]
can be written as a subtraction-free expression. 
\iffalse
Specifically, for any integer $r>9N^2$
%$r>\frac{2}{1-\cos(\alpha)}$,
we have
\begin{equation}
\label{eq:polya} F_N(x)=\frac{p(x)}{(1+x)^r}
\end{equation}
where $p(x)$ is a polynomial with nonnegative coefficients.
\fi
Furthermore, if $F_N(x) Q(x)=P(x)$ where $P(x)$ is a polynomial with nonnegative coefficients,
then $\deg(P)>N$. 
\end{lemma}

\begin{proof}
By a classical theorem of P\'olya~\cite{Polya}, 
the fact that $F_N(x)>0$ for any $x\ge 0$ (actually, any~$x\in\RR$) implies that we can write $F_N(x)=p(x)/(1+x)^r$
for $r$ a sufficiently large integer, and~$p(x)$ a polynomial with nonnegative coefficients.
(It can be shown that $r>9N^2$ suffices, cf.\ \cite[p.~222]{Powers-Reznick}.)

Let us prove the second statement. 
Assume that on the contrary, $\deg(P)\le N$, and denote
$P(x)=\sum_{k=0}^N p_k x^k$.
Let $u=1+\sqrt{-1}/N$ and $v=1-\sqrt{-1}/N$ be the roots~of~$F_N$.
Then 
\[
u^k+v^k=2\biggl(1+\frac{1}{N^2}\biggr)^{k/2} \cos\biggl(k\cdot \tan^{-1}\biggl(\frac{1}{N}\biggr)\biggr). 
\]
If $0\le k\le N$, then $0\le k\tan^{-1}\bigl(\frac{1}{N}\bigr)\le \frac{k}{N}\le 1<\frac{\pi}{2}$, 
implying that $u^k+v^k>0$. 
Consequently
\[
0=F_N(u) Q(u)+F_N(v) Q(v)
=P(u)+P(v)
=\sum_{k=0}^N p_k (u^k+v^k)>0, 
\]
a contradiction. 
\end{proof}

\begin{proposition}
\label{prop:f_n-lower-bound} 
The subtraction-free complexity of the
univariate polynomial %$f_n(x)=F_N^2(x)$
\[
G_n(x)=F_{2^{2^n}}(x)=(x-1)^2+2^{-2^{n+1}}, 
\]
while finite, is greater than $2^n$.
The ordinary arithmetic circuit complexity of $G_n(x)$ is $O(1)$ if arbitrary constants are allowed as
inputs. If $1$ is the only input constant allowed, the ordinary complexity of $G_n(x)$ is~$O(n)$. 
\end{proposition}

\begin{proof}
By Lemma~\ref{lem:F_alpha-is-SF}, 
the subtraction-free complexity of $G_n$ is finite,
and for any representation $G_n=P/Q$ where $P$ and $Q$ are polynomials with 
nonnegative coefficients, we have $\deg(P)>2^{2^n}$. 
Now Lemma~\ref{lem:denom-to-complexity} implies that 
subtraction-free complexity of~$G_n$ is greater than~$2^n$. 
Finally, the last statement of the proposition follows from the fact that $2^{2^n}$ can be computed
by iterated squaring. 
\end{proof}

\iffalse
We are now prepared to give an example where the gap between ordinary and
sub\-trac\-tion-free complexity is exponential.
Let us set $N=2^{2^{n-1}}$ and
\begin{equation}
\label{eq:F-alpha(n)}
f_n=F_N^2=((x-1)^2+2^{-2^n})^2.
\end{equation}
 The ordinary arithmetic circuit complexity
of $f_n(x)$ is $O(1)$ provided arbitrary constants are allowed as
inputs. Even if $1$ is the only input constant allowed, the circuit
complexity of $f_n(x)$ is $O(n)$ because $2^{2^n}$ can be computed
by iterated squaring.

In view of Lemma~\ref{lem:F_alpha-is-SF}, it makes sense to talk
about subtraction-free complexity of~$F_N(x)$ and $f_N(x)$.

\begin{proof}
Combining Lemmas~\ref{lem:denom-to-complexity} and~\ref{lem:F_alpha-is-SF},
we conclude that
%if $\alpha<\frac{\pi}{N+2}$, then
subtraction-free complexity of $F_N^2(x)$ is greater or equal
than~$\log_2(N)$.
%For $\alpha=\alpha(n)\sim
%\operatorname{const}\cdot 2^{-2^n}$, we can use $N\sim
%\operatorname{const}\cdot 2^{2^n}$, and arrive at the desired lower
%bound.
\end{proof}
\fi

%We note that Lemma~\ref{lem:F_alpha-is-SF} entails that the
%subtraction-free complexity of $f_n$ is finite and moreover, that
%the bound on the degree of subtraction-free representations of $f_n$
%is close to sharp due to Lemma~\ref{lem:denom-SF}.

The reader might feel uncomfortable about the fact that the polynomial~$G_n(x)$
in Proposition~\ref{prop:f_n-lower-bound} has a coefficient whose binary notation has exponential length.
To alleviate those concerns, we present a closely related example
that does not have this drawback.
In doing so, we use a modification of the well-known Lazard-Mora-Philippon trick,
cf., e.g.,~\cite{Grigoriev-Vorobjov}.

\begin{proposition}
\label{prop:g_n-lower-bound} 
Define the homogeneous polynomials
$H_n(t,x_1,\dots,x_n)%\in\ZZ[t,x_1,\dots,x_n]
$ by
\begin{align*}
H_n(t,x_1,\dots,x_n)=&(x_1-t)^4+(x_1-2 x_2)^4\\
&+(x_2^2-t x_3)^2+(x_3^2-t x_4)^2
+\cdots+ (x_{n-1}^2-t x_n)^2+4(x_1-t)^2 x_n^2+2 x_n^4\,.
\end{align*}
%
%$g_n=g_n(x_1,\dots,x_n)\in\ZZ[x_1,\dots,x_n]$ by
%\[
%g_n(x_1,\dots,x_n)
%g_n=(1-x_1)^2 + (x_1 - 2x_2)^2 + (x_2^2 - x_3)^2 + (x_3^2 - x_4)^2 + \cdots +
%(x_{n-1}^2 - x_n)^2 + 4x_n^2 x_1\,.
%\]
Then the subtraction-free complexity of $H_n$, while finite, is greater than~$2^{n-2}$. 
%greater than $c\cdot 2^n$, with $c>0$ an absolute constant.
By contrast,
the ordinary arithmetic circuit complexity of $H_n$ is linear
in~$n$.
\end{proposition}

\begin{proof}
%First we note that $g_n$ can be obtained from the homogeneous polynomial
%\[
%H_n=t^2(t-x_1)^2 + t^2(x_1 - 2x_2)^2 + (x_2^2 - tx_3)^2 + (x_3^2 - tx_4)^2 + \cdots +
%(x_{n-1}^2 - t x_n)^2 + 4t x_n^2 x_1
%\]
%by setting $t=1$.
Since $H_n(t,x_1,\dots,x_n)$ is positive for any nonnegative (in fact, any real) vector
$(t,x_1,\dots,x_n)\neq (0,0,\dots,0)$, % except for $x_1=\cdots=x_n=0$,
%non-negative values of the arguments,
P\'olya's theorem~\cite{Polya} tells us that we can write
\[
H_n(t,x_1,\dots,x_n)=p(t,x_1,\dots,x_n)/(t+x_1+\cdots+x_n)^r, 
\]
for some polynomial $p$ with nonnegative coefficients, and some positive integer~$r$. 
So the subtraction-free complexity of $H_n$ is finite.

Assume that $H_n=P/Q$ where $P$ and $Q$ are polynomials  with nonnegative coefficients. 
Substituting $t=1,x_2=2^{-1},x_3=2^{-2},\dots,x_n=2^{-2^{n-2}}$, we get:
\begin{align*}
\frac{P(1,x_1,2^{-1},2^{-2},\dots,2^{-2^{n-2}})}{Q(1,x_1,2^{-1},2^{-2},\dots,2^{-2^{n-2}})}
&=H_n(1,x_1,2^{-1},2^{-2},\dots,2^{-2^{n-2}})\\
&=(x_1-1)^4+(x_1-1)^4+4(x_1-1)\cdot 2^{-2^{n-1}}+2\cdot 2^{-2^n}\\
&=2(F_{2^{2^{n-2}}}(x_1))^2. 
\end{align*}
Since $P(1,x_1,2^{-1},2^{-2},\dots,2^{-2^{n-2}})$ is a polynomial with nonnegative coefficients,
we can apply Lemma~\ref{lem:F_alpha-is-SF}
to conclude that $\deg(P)\ge \deg_{x_1}(P)>2^{2^{n-2}}$. 
Now Lemma~\ref{lem:denom-to-complexity} implies that 
the subtraction-free complexity of $H_n$ is greater than~$2^{n-2}$. 
%
%To obtain the lower bound for the subtraction-free complexity of~$g_n$,
%we set $x_{k+2}=2^{-2^k}$ for $0 \le k \le n-2$ to get
%We next observe that
%\begin{equation*}
%\label{eq:gn-specialized}
%g_n(x_1,2^{-1},2^{-2},2^{-4},2^{-8},\dots,2^{-2^{n-2}})/2
%=(1-x_1)^2 + 2\cdot 2^{-2^{n-1}}x_1
%=f_{n-1}(x_1).
%\end{equation*}
%Now Proposition~\ref{lem:denom-SF} implies that any representation
%\[
%g_n(x_1,2^{-1},2^{-2},\dots)=p_n(x_1)/q_n(x_1)
%\]
%in which polynomials $p_n$ and $q_n$ have  nonnegative coefficients has
%$\deg(q_n)>\operatorname{const}\cdot 2^{2^n}$.
%Hence the same is true for~$g_n$ (taking the degree with respect to~$x_1$),
%and the argument used in the proof of Lemma~\ref{lem:denom-to-complexity}
%implies the desired complexity bound.
\end{proof}

%\newpage

\subsection*{Acknowledgements}
We thank Leslie Valiant for bringing the paper~\cite{Jerrum-Snir} to our attention.

\pagebreak[3]

%\vspace{-.08in}

\end{document}